\documentclass[11pt]{amsart}

\usepackage{amsmath,amssymb,amsfonts,amsthm}
\usepackage{mathtools}
\usepackage{amsmath}
\usepackage{graphics}
\usepackage{tikz-cd}
\newtheorem{theorem}{Theorem}[section]
\newtheorem{corollary}[theorem]{Corollary}
\newtheorem{lemma}[theorem]{Lemma}
\newtheorem{proposition}[theorem]{Proposition}

\newtheorem{definition}[theorem]{Definition}
\newtheorem{remark}[theorem]{Remark}
\newtheorem{example}[theorem]{Example}

\title[On spherical Milnor spaces I]{On spherical Milnor Classifying Spaces I:\\
 differential geometry}

\author{Jean-Pierre Magnot}

\address{{SFR MATHSTIC, LAREMA, Universit\'e d’Angers, 2 Bd Lavoisier, 
49045 Angers cedex 1, France;  Lyc\'ee Jeanne d'Arc, 40 avenue de Grande Bretagne, 63000 Clermont-Ferrand, 
France}; Lepage Research Institute, 17 novembra 1, 081 16 Presov, Slovakia}
\email{\small magnot@math.cnrs.fr; jean-pierr.magnot@ac-clermont.fr; jp.magnot@gmail.com}


\begin{document}

\begin{abstract}
We compare the original product-subset diffeology on the Milnor join
of Magnot--Watts with the locally finite final diffeology on a signed
spherical model. The canonical squaring map from the spherical space
to the Magnot--Watts join is smooth, whereas the positive square-root
section is generally not; we also construct a different smooth
injection from the locally finite simplicial model whose inverse is
not smooth.

For Hausdorff Fr\'echet--Lie groups whose smooth real-valued
functions separate points, we compute the full algebraic
internal tangent space, exhibiting separate first-order branches
labelled by elements of the group at every inactive coordinate.
The spherical weighted energy descends plotwise and restricts to a
Riemannian metric on each support stratum. Smooth normalized strictly positive-definite
kernels yield non-canonical ambient tangent metrics; compact Lie
groups admit an explicit realization through their heat kernels. We characterize global diffeological forms
by descent, construct plotwise Hausdorff Jacobians and varying-degree
complementary forms, and show why these data do not define a canonical
Hodge star across support changes. Stratumwise Hodge--Dirac operators
remain available, but an explicit example obstructs their automatic gluing.

We prove smooth contractibility of the total spherical space for any
diffeological group, as well as equivariant uniqueness of maps to it.
For any diffeological group we establish a classifying theorem
for $D$-numerable principal bundles, without regularity hypotheses.
For Lie groups we construct barycentric connections and
their Chern--Weil forms, and distinguish sign local systems from gerbes. The hypotheses and proofs
specify which conclusions are global, stratumwise, or conditional.
\end{abstract}

\maketitle

\section{Introduction}

The construction of classifying spaces plays a central role in geometry and topology, notably in the theory of principal bundles, characteristic classes, and differential forms \cite{Milnor1956I,Milnor1956II,BottTu1982}. Among the various models, Milnor's construction provides a particularly flexible approach, based on infinite joins and barycentric coordinates. Its extension to generalized smooth settings is naturally related to diffeology, introduced by Souriau \cite{Souriau1980} and systematically developed by Iglesias-Zemmour \cite{IglesiasZemmour2013}, as well as to other categories of smooth spaces and convenient frameworks \cite{BaezHoffnungRogers2010,Stacey2011}.

In the diffeological setting, classifying spaces and their smooth structures have been studied from several complementary viewpoints. The \(D\)-topology and smooth classifying spaces were investigated in \cite{ChristensenSinnamonWu2014,ChristensenWuSmooth}, while tangent constructions and tangent bundles for diffeological spaces were developed in \cite{ChristensenWu2016}. The diffeology of Milnor's classifying space was analyzed in \cite{MagnotWatts2017}, and diffeological covariant derivatives were studied in \cite{MagnotConnections}. Related aspects of diffeological geometry, including global and differential structures, also appear in \cite{GMW2023,Magnotloops2019,Savelyev2026,Scalbi2024,Watts2012Thesis}.

The present work continues this line of investigation by developing a systematic study of generalized Milnor classifying spaces endowed with additional geometric structures. The key idea is to replace the usual simplicial barycentric coordinates by spherical or projective coordinates, so that the normalization becomes quadratic. This point of view is also naturally related to information-geometric constructions, where square-root coordinates and Fisher--Rao type metrics play a central role \cite{AmariNagaoka2000,AyJostLeSchwachhofer2017,Rao1945}. The relationship of these constructions with information geometry is outlined in an appendix.

A central part of the paper is devoted to differential-geometric structures on these generalized Milnor spaces. We construct a compatible plotwise quadratic energy derived from spherical coordinates and genuine Riemannian metrics on individual support strata. The global diffeological de Rham complex is described by descent through the finite spherical quotients. Revisiting plotwise Hausdorff volumes and forms of formal degree $\infty-q$, we identify both their valid regular-plot interpretation and their failure to provide a canonical global Hodge star. Formal Hodge adjoints and canonical de Rham--Dirac operators are obtained on support strata when $G$ is finite-dimensional. An explicit counterexample shows that the stratumwise Laplacians need not glue to the full quotient; no ambient Hodge or Dirac operator is asserted without further analytic and tangent-bundle data. 

Another important aspect concerns twisted and higher geometric structures. The spherical and projective constructions naturally lead to \(\mathbb Z_2\)-actions, local systems, and twisted objects. Such structures are closely related to extensions of infinite-dimensional Lie groups \cite{Neeb2004,Neeb2007}, bundle gerbes and non-abelian gerbes \cite{Breen2008,BreenMessing2005,Murray1996,NikolausWaldorf2013,Waldorf2007}, and higher categorical or higher bundle structures \cite{BaezLauda2004,Wockel2008}. They also connect with index-theoretic and field-theoretic constructions involving gerbes and determinant-type objects \cite{CareyMickelssonMurray1997,Freed1995}. This first paper treats the analytic and higher-geometric constructions under explicit hypotheses, distinguishing results proved here from problems left open for subsequent work.
The main results of this article can be summarized as follows:
\begin{itemize}
\item we distinguish the product-subset Magnot--Watts join from the locally finite spherical and simplicial colimits, giving the exact smooth comparison maps;
\item we calculate the internal tangent vector spaces for smoothly point-separating Fr\'echet--Lie groups, including a local cone lemma, group-labelled activation branches, and explicit retractions;
\item we prove plotwise metric compatibility and construct conditional ambient tangent metrics from smooth positive kernels, with a heat-kernel Hilbert realization for compact groups;
\item we introduce plotwise Hausdorff volumes and complementary-degree forms, prove obstructions to a globally natural Hodge star, and construct Hodge and Clifford--Dirac operators on support strata;
\item we prove a classifying theorem for arbitrary diffeological groups and $D$-numerable bundles, and construct weighted connections separately for Lie groups;
\item we give rigorous conditional constructions for twisted and fiberwise Dirac families and an operator-defect cocycle.
\end{itemize}

These results identify intrinsic diffeological and algebraic tangent structures,
including optional metric extensions, and separate them from Hodge and
Clifford--Dirac data that exist on individual smooth strata. They are intended as a foundation for further developments, in particular for the study of integrable systems and geometric flows.

The paper is organized as follows. Section~2 recalls diffeological preliminaries and Milnor constructions; Sections~3--5 concern spherical models, their diffeology, and tangent phenomena. Section~6 establishes the chartwise energy and stratumwise Riemannian metrics. Section~7 describes the global diffeological de Rham complex, plotwise volumes and complementary-degree forms; Section~8 explains the failure of a canonical global Hodge star and develops Hodge theory on support strata; and Section~9 concerns Clifford and Dirac operators. Section~10 concerns sign local systems and higher-geometric data; Section~11 establishes the spherical principal-bundle and weighted-connection constructions and studies classical groups; Section~12 discusses diffeomorphism groups, vertical Dirac families and the operator-defect cocycle. The appendix specifies the domain of the Fisher--Rao comparison.

\section{Diffeological preliminaries and Milnor classifying spaces}
\label{sec:diff-prelim}

\subsection{Diffeological spaces}

We recall the basic notions from diffeology that will be used throughout the paper. General references are \cite{IglesiasZemmour2013} and the survey \cite{GMW2023}. Tangent structures are mainly taken from \cite{ChristensenWu2016,GMW2023}, while pseudo-bundles follow \cite{Pervova2016,Magnot-Cauchy,MagnotConnections}. Related references include \cite{Pervova2017,Pervova2019}.

Diffeology provides a flexible framework for generalized smooth spaces by replacing atlases with families of parametrizations. This makes it possible to treat manifolds, infinite-dimensional spaces, quotients, and singular spaces in a common smooth category.

\begin{definition}[Diffeology]
Let \(X\) be a set. A \emph{parametrization} of \(X\) is a map
\[
p:U\to X,
\]
where \(U\) is an open subset of some Euclidean space. A \emph{diffeology} \(\mathcal P\) on \(X\) is a set of parametrizations, called \emph{plots}, satisfying:
\begin{enumerate}
\item constant maps are plots;
\item plots satisfy the locality condition;
\item plots are stable under smooth reparametrization.
\end{enumerate}
A set equipped with a diffeology is called a \emph{diffeological space}.
\end{definition}

\begin{definition}[Smooth maps]
Let \((X,\mathcal P_X)\) and \((Y,\mathcal P_Y)\) be diffeological spaces. A map
\[
F:X\to Y
\]
is \emph{smooth} if \(F\circ p\in\mathcal P_Y\) for every plot \(p\in\mathcal P_X\).
\end{definition}

The category of diffeological spaces is complete, cocomplete, and forms a quasi-topos \cite{BaezHoffnungRogers2010}. We shall use standard constructions such as pull-back, push-forward, product, subset, and quotient diffeologies \cite{Sou,IglesiasZemmour2013}.

\begin{definition}[Pull-back and push-forward diffeologies]
Let \(f:X\to X'\) be a map.
\begin{itemize}
\item If \(X'\) is diffeological, the pull-back diffeology on \(X\) is
\[
f^*(\mathcal P')=\{p:O_p\to X\mid f\circ p\in\mathcal P'\}.
\]
\item If \(X\) is diffeological, the push-forward diffeology on \(X'\) is the coarsest diffeology making all maps \(f\circ p\) plots.
\end{itemize}
A map \(f:X\to X'\) is called a \emph{subduction} if the diffeology of \(X'\) is the push-forward of the diffeology of \(X\).
\end{definition}

\begin{definition}[Product, subset and quotient diffeologies]
Products are endowed with the finest diffeology making all projections smooth. Subsets are endowed with the subset diffeology. Quotients are endowed with the push-forward diffeology induced by the quotient projection.
\end{definition}

When \(X\) is a finite- or infinite-dimensional smooth manifold modeled on a complete locally convex space, it carries its natural \emph{nebulae diffeology}, consisting of the usual smooth parametrizations from open subsets of Euclidean spaces.

\subsection{Frölicher spaces}

We shall occasionally use Frölicher spaces as an intermediate framework between manifolds and diffeological spaces; see \cite{KM,Ma2006-3,Watts2012Thesis,BIgKWa2014}.

\begin{definition}[Frölicher space]
A \emph{Frölicher space} is a triple
\[
(X,\mathcal F,\mathcal C),
\]
where \(\mathcal C\) is a set of contours \(c:\mathbb R\to X\) and \(\mathcal F\) is a set of functions \(f:X\to\mathbb R\), determined by the mutual conditions:
\[
f\in\mathcal F \iff f\circ c\in C^\infty(\mathbb R,\mathbb R)
\quad\text{for all } c\in\mathcal C,
\]
and
\[
c\in\mathcal C \iff f\circ c\in C^\infty(\mathbb R,\mathbb R)
\quad\text{for all } f\in\mathcal F.
\]
\end{definition}

Every Frölicher space carries a natural diffeology \(\mathcal P_\infty(\mathcal F)\), whose plots are the maps \(p:O\to X\) such that \(f\circ p\) is smooth for all \(f\in\mathcal F\).

\begin{proposition}[\cite{Ma2006-3}]
A map between Frölicher spaces is smooth in the Frölicher sense if and only if it is smooth for the associated nebulae diffeologies.
\end{proposition}

Thus one has the informal hierarchy
\[
\text{smooth manifolds}
\;\Longrightarrow\;
\text{Frölicher spaces}
\;\Longrightarrow\;
\text{diffeological spaces}.
\]

Conversely, every diffeological space admits a canonical Frölicher completion generated by its smooth real-valued functions \cite{Watts2012Thesis,BIgKWa2014}. On Fr\'echet spaces, convenient smoothness and Bastiani smoothness agree under the standard hypotheses of the convenient calculus \cite{KM}. Mere existence of directional (G\^ateaux) derivatives, without continuity and smooth dependence conditions, is not sufficient.

\subsection{Diffeological groups and pseudo-bundles}

A \emph{diffeological vector space} is a vector space equipped with a diffeology for which addition and scalar multiplication are smooth. A \emph{diffeological group} is a group equipped with a diffeology for which multiplication and inversion are smooth \cite{IglesiasZemmour2013,Pervova2016}.

\begin{definition}[Diffeological fiber pseudo-bundle]
Let \(E\) and \(X\) be diffeological spaces. A smooth surjective map
\[
\pi:E\to X
\]
is called a \emph{diffeological fiber pseudo-bundle} if \(\pi\) is a subduction.
\end{definition}

This notion generalizes both diffeological vector pseudo-bundles \cite{Pervova2016} and Souriau's quantum structures \cite{Sou}. No local triviality is assumed, and no fixed typical fiber is required.

\begin{definition}[Vector pseudo-bundle]
Let \(\mathbb K\) be a diffeological field. A diffeological fiber pseudo-bundle
\[
\pi:E\to X
\]
is a \(\mathbb K\)-vector pseudo-bundle if each fiber \(E_x\) is a diffeological \(\mathbb K\)-vector space and if fiberwise addition and scalar multiplication are smooth.
\end{definition}

Fiberwise isomorphism of the fibers does not imply local triviality. We reserve the term \emph{diffeological vector bundle} for a locally trivial diffeological vector pseudo-bundle with a fixed typical fiber.

\subsection{Tangent and cotangent structures}

There are several non-equivalent notions of tangent space in diffeology. The review \cite{GMW2023} identifies five main constructions, originating either from germs of paths or from derivations of smooth functions.

The \emph{internal tangent cone} is defined using germs of paths and generalizes the kinematic tangent construction \cite{DN2007-1,KM}. Its vector-space completion gives the \emph{internal tangent space}, studied for diffeological spaces in \cite{He1995,ChristensenSinnamonWu2014}. Another construction, the \emph{external tangent space}, is defined using derivations \cite{Ma2013,GW2020}. These tangent spaces coincide for finite-dimensional manifolds but differ in general diffeological spaces \cite{ChristensenWu2016,GMW2023}.

The cotangent space and the algebra of differential forms are less ambiguous: differential forms are defined plotwise, by assigning to each plot a classical differential form on its domain, compatibly with smooth reparametrizations \cite{IglesiasZemmour2013,Ma2013}.

\subsection{Diffeological Lie groups}

For a diffeological group \(G\), the internal tangent cone at the identity is often a vector space in the classical examples. However, the existence of a canonical Lie algebra structure in full generality remains subtle. We shall work with diffeological groups for which the tangent space at the identity is equipped with a compatible Lie algebra structure; these will be called \emph{diffeological Lie groups}. For technical aspects, see \cite{Les}.

\subsection{Differential forms with values in pseudo-bundles}

Let \(\pi:E\to X\) be a diffeological vector pseudo-bundle. An \(E\)-valued differential \(n\)-form on \(X\) is defined plotwise: for every plot \(p:O_p\to X\), one assigns a smooth fiberwise alternating map
\[
\alpha_p:\wedge^n TO_p\to p^*E,
\]
compatible with smooth reparametrizations and with changes of plots. The resulting space is denoted
\[
\Omega^n(X,E),
\qquad
\Omega(X,E)=\bigoplus_{n\geq 0}\Omega^n(X,E).
\]

The space \(\Omega^n(X,E)\) is endowed with the functional diffeology: a parametrized family of forms is smooth if its evaluation on every plot is smooth. This generalizes the usual diffeology on scalar differential forms \cite{Ma2013}.

One must be careful with wedge products and exterior differentials in this setting. A wedge product requires a smooth fiberwise multiplication on \(E\), and a natural exterior differential for \(E\)-valued forms generally requires additional connection data.

\subsection{Connections and covariant derivatives}

For principal bundles, a connection is classically described by a connection form satisfying the usual equivariance condition. In the diffeological setting, such connection forms can be defined on suitable frame bundles, but this requires care because vector pseudo-bundles need not have a fixed typical fiber and because the diffeology of the frame bundle may depend on the chosen construction.

For this reason, we shall mainly use covariant derivatives, defined plotwise.

\begin{definition}[Covariant derivative]
Let \(\pi:E\to X\) be a diffeological vector pseudo-bundle. A \emph{covariant derivative} on \(E\) is a family
\[
(p^*\nabla)_{p\in\mathcal P_X}
\]
such that for every plot \(p:O_p\to X\), \(p^*\nabla\) is a covariant derivative on the pullback bundle \(p^*E\to O_p\), and these covariant derivatives are compatible with smooth reparametrizations of plots.
\end{definition}

We denote by \(\mathcal{CD}(E)\) the space of covariant derivatives on \(E\).

The endomorphism pseudo-bundle $\operatorname{End}(E)$ is endowed with a functional diffeology ensuring smooth evaluation. We assume, as part of the plotwise connection convention, that the difference of two admissible pullback connections is represented by a smooth $\operatorname{End}(E)$-valued plotwise one-form; this is automatic for ordinary locally trivial vector bundles and is the framework used in \cite{Magnot-Cauchy,MagnotConnections}.
\begin{theorem}[Affine structure of plotwise connections]
If $\mathcal{CD}(E)$ is non-empty under this convention, then it is an affine space modeled on
\[
\Omega^1(X,\operatorname{End}(E)).
\]
Its affine operations are smooth for the corresponding functional diffeologies.
\end{theorem}
\begin{proof}
Fix $\nabla^0\in\mathcal{CD}(E)$. If $\nabla^1$ is another compatible family, then for every plot $p:U\to X$, the difference
$A_p=p^*\nabla^1-p^*\nabla^0$ is $C^\infty(U)$-linear in the differentiated section. The Leibniz rule for both connections cancels the derivative of the scalar multiplier. Consequently $A_p$ is a smooth one-form on $U$ with values in fiberwise endomorphisms of $p^*E$. The compatibility of the two connections under reparametrization gives $A_{p\circ f}=f^*A_p$. By the stated endomorphism convention, $(A_p)_p$ is an element of $\Omega^1(X,\operatorname{End}(E))$. Conversely, for such a form $A$, setting $(p^*\nabla^A)=(p^*\nabla^0)+A_p$ gives a covariant derivative on each plot, satisfies the Leibniz rule and is compatible with reparametrizations. These operations are inverse, identifying the connection space with the indicated affine torsor. Smoothness of affine addition and subtraction follows plotwise from smooth evaluation and the functional diffeology.
\end{proof}

\subsection{Pseudo-metrics and Riemannian structures}

Since tangent objects in diffeology are often pseudo-bundles rather than genuine vector bundles, Riemannian structures are naturally formulated as pseudo-metrics.

\begin{definition}
A \emph{pseudo-metric} on a diffeological vector pseudo-bundle \(V\to X\) is a smooth symmetric fiberwise bilinear form
\[
g:V\times_X V\to\mathbb R.
\]
\end{definition}

This is the standard substitute for Riemannian metrics in diffeology and is used in the study of Levi-Civita connections and de Rham-type operators \cite{Pervova2017,GW2020}.

Given a pseudo-metric \(g\) and a covariant derivative \(\nabla\), one defines covariant differentiation of sections along tangent directions whenever the chosen tangent model permits such a contraction. In the present article, the relevant metrics and connections will first be constructed on finite-index generating models (which are still infinite-dimensional when $G$ is Fr\'echet) and then checked for compatibility with the quotient diffeology. This avoids ambiguity at the level of abstract tangent functors.

\subsection{The original Magnot--Watts construction}
\label{subsec:mw-correct}
A comparison of diffeologies must distinguish the \emph{pointwise finite}
join of \cite{MagnotWatts2017} from the colimit of finite joins used
below.  Let
\[
 \widetilde E_{\rm MW}(G)=\{(t_i,g_i)_{i\geq 1}\in
 [0,1]^{\mathbb N}\times G^{\mathbb N}:t_i\geq0,
 \ \textstyle\sum_i t_i=1,\ \#\{i:t_i\ne0\}<\infty\}.
\]
Give this set the \emph{subset diffeology of the infinite product}.
Identify $(t_i,g_i)\sim(t_i,g_i')$ if $g_i=g_i'$ for all
$i$ such that $t_i>0$, and give $E_{\rm MW}(G)$ the resulting
quotient diffeology. In \cite[Definition~3.2]{MagnotWatts2017},
the diagonal action is written using right multiplication by
$h^{-1}$. We use left multiplication by $h$, which is equivalent
through coordinatewise group inversion but should not be confused
with the literal formula in the reference.

\begin{proposition}[Pointwise finite is not locally finite]
\label{prop:mw-not-final}
The original Magnot--Watts join has plots whose barycentric
supports are not locally contained in any fixed finite set.
Consequently its diffeology differs from the final diffeology
generated by the finite-dimensional joins
$\Delta_I\times G^I$ with the ordinary subset diffeology on
$\Delta_I$.
\end{proposition}
\begin{proof}
Choose $\rho\in C_c^\infty((-1,1),[0,1])$ with $\rho(0)=1$.
For $n\geq3$, let $t_n(s)=e^{-n^2}\rho(n^3(s-1/n))$.
The supports are pairwise disjoint for all sufficiently large $n$;
change finitely many of them if necessary so this holds for all
$n\geq3$. Set
\[
 t_1(s)=1-\sum_{n\geq3}t_n(s),\qquad t_2(s)=0.
\]
At each $s$ the sum contains at most one term. For every fixed
integer $r\geq0$, the $r$-th derivative of $t_n$ is bounded by
$e^{-n^2}n^{3r}\|\rho^{(r)}\|_\infty$; this tends to zero faster
than any power of $1/n$. Therefore the piecewise sum, including
its value zero at $s=0$, is smooth to all orders. The coefficient
$t_1$ is nonnegative since $e^{-n^2}<1$, and the weights add to
one. Give all group coordinates the constant value $e$.
This is a plot of the product-subset generating space and hence
of $E_{\rm MW}(G)$. At $s=0$ the support is $\{1\}$, yet every
neighbourhood of zero contains $1/n$ with $t_n(1/n)>0$ for
arbitrarily large $n$. Such a plot cannot factor locally through
any fixed finite-dimensional join. The final diffeology has the
latter local-lifting property by definition.\end{proof}

For comparison define $E_{\rm LF}(G)$ using the same underlying
formal barycentric sums, but with the \emph{final diffeology}
from $\Delta_I\times G^I$ for finite $I$. There is a natural
smooth map
\[
 u:E_{\rm LF}(G)\longrightarrow E_{\rm MW}(G)
\]
which is the identity on the underlying set: every finite-chart
plot is a product-subset plot. Proposition~\ref{prop:mw-not-final}
shows that the reverse identity need not be smooth.

\begin{theorem}[Magnot--Watts, quoted with its hypotheses]
\label{thm:MW}
For any diffeological group $G$, the original construction
$E_{\rm MW}(G)\to B_{\rm MW}(G)$ is a weakly
$D$-numerable principal diffeological bundle
\cite[Proposition~3.4]{MagnotWatts2017}, and its total space
is smoothly contractible \cite[Proposition~5.1]{MagnotWatts2017}.
For Hausdorff, second-countable, smoothly paracompact
bases, the classification theorem of
\cite[Theorem~3.6]{MagnotWatts2017} identifies isomorphism
classes of $D$-numerable principal $G$-bundles with suitable
smooth homotopy classes of maps to $B_{\rm MW}(G)$.
\end{theorem}
The classification statement is an imported result, not an assertion
about $E_{\rm LF}(G)$ or the spherical construction. A weighted
connection is developed in \cite[Section~4]{MagnotWatts2017}
under the regularity hypotheses of that paper. The corresponding
spherical connection is proved directly in
Theorem~\ref{thm:spherical-connection} below.

\section{Spherical and projective models}\label{3}
\subsection{Spherical coordinates and the Milnor comparison}
For finite nonempty $I\subset\mathbb N$ set
\[
 S_I=\{x\in\mathbb R^I:\sum_{i\in I}x_i^2=1\},\qquad
 \Delta_I=\{t\in[0,\infty)^I:\sum_{i\in I}t_i=1\}.
\]
Each carries its ordinary manifold or subset diffeology. The
squaring map $q_I:S_I\to\Delta_I$, $q_I(x)=(x_i^2)$,
is smooth, as all its components are polynomials. The positive
square-root section $r_I(t)=(\sqrt{t_i})$ is continuous but
not smooth on the closed simplex: for $|s|<1$, the plot
\[
 s\longmapsto(1-s^2,s^2)\in\Delta_{\{1,2\}}
\]
would be mapped to $(\sqrt{1-s^2},|s|)$, whose second
component is not differentiable at zero.
There is no first-order paradox: on open sign chambers, the maps
$q_I$ are diffeomorphisms onto the open simplex and the problem
occurs precisely when a coordinate is zero.

\subsection{Spherical total space and actions}
Let $G$ be a diffeological group. Set
\[
 \widetilde E^{\rm sph}(G)=\{(x_i,g_i)_{i\in\mathbb N}:x\in
 \mathbb R^{(\mathbb N)},\ \sum_i x_i^2=1,\ g_i\in G\},
\]
where $\mathbb R^{(\mathbb N)}$ consists of finitely supported
sequences. Declare $(x_i,g_i)\sim(x_i,g'_i)$ precisely when
$g_i=g'_i$ for every index with $x_i\ne0$. Define
$E^{\rm sph}(G)=\widetilde E^{\rm sph}(G)/\sim$ and equip it with
the final diffeology generated by the maps
$\psi_I:S_I\times G^I\to E^{\rm sph}(G)$, using extension by
zero outside $I$ and forgetting inactive group labels.

\begin{proposition}
Left diagonal multiplication by $G$ and simultaneous sign change by
$\mathbb Z_2$ define smooth commuting actions on $E^{\rm sph}(G)$.
The successive quotients in either order agree diffeologically with
the quotient by $G\times\mathbb Z_2$.
\end{proposition}
\begin{proof}
Each action lifts smoothly to every generating chart; the lifts
preserve the equivalence relation. Commutation is immediate on
representatives. The assertion about quotient diffeologies follows
from the universal property of final diffeologies.
\end{proof}

\begin{remark}
The classifying property is not a formal consequence of
commutation; it is established for all diffeological groups and
$D$-numerable bundles in Theorem~\ref{thm:spherical-classification}.
The double quotient classifies $G\times\mathbb Z_2$, not $G$.
\end{remark}

\subsection{Comparison maps, smooth injections, and their limits}
\label{subsec:comparison-precise}
Define the finite-support squaring map on formal sums by
\[
 \Phi:E^{\rm sph}(G)\longrightarrow E_{\rm MW}(G),\qquad
 [(x_i,g_i)]\longmapsto[(x_i^2,g_i)].
\]
\begin{theorem}[Correct direction of the canonical comparison]
\label{thm:comparison}
The map $\Phi$ is well-defined, surjective, $G$-equivariant
and smooth. It has a set-theoretic positive square-root section,
but this section is not smooth even when $G$ is trivial.
Moreover, a support-preserving map from the original
$E_{\rm MW}(G)$ to our locally finite spherical total space
cannot be smooth if it sends each formal sum to a formal sum
with exactly the same support.
\end{theorem}
\begin{proof}
Squaring preserves the set of active indices and hence respects
the relation forgetting inactive labels. Any plot $p$ of the
spherical space locally lifts through a generating map
$S_I\times G^I$; the composite with $\Phi$ has a lift to the
infinite-product generating set for the original Milnor space
with weights $x_i^2$ and constant inactive labels. The product
coordinates are smooth, so $\Phi\circ p$ is a plot. Every
barycentric formal sum has its positive square root as a set,
which proves surjectivity. Non-smoothness of that section follows
from the preceding two-coordinate example with all labels fixed.
Finally take the product-subset plot with non-locally bounded
support constructed in Proposition~\ref{prop:mw-not-final}.
A support-preserving image has the same unbounded supports near
$s=0$, whereas every spherical plot has locally bounded support.
This contradiction proves the last claim. Equivariance is immediate
because the transformations leave the weights unchanged.
\end{proof}

\begin{remark}[Endpoint-stationary homotopies versus the Magnot--Watts diffeology]
\label{rem:mw-endpoint-stationarity}
The notion of \emph{stationarity at the ends of a parameter interval}
should not be confused with the diffeology of the join.  In
\cite[Definition~3.1]{MagnotWatts2017}, the intervals $[0,1]$
carry their ordinary subset diffeology; the generating space of
the join has the subset diffeology inherited from their infinite
product. Thus the coordinate curve
\[
    s\longmapsto (1-s^2,s^2),\qquad |s|<1,
\]
is a plot of the two-coordinate simplex. By contrast, in the
proof of \cite[Proposition~3.16]{MagnotWatts2017}, one chooses a
\emph{particular} smooth reparametrization $b$ which is constant
near the endpoints of $[0,1]$, so that piecewise homotopies
can be concatenated smoothly.  This does not impose an endpoint
stationarity condition on all barycentric plots.

More precisely, along a spherical curve with $x_k(0)=0$
and $x_k'(0)\ne0$, the function $t_k(s)=x_k(s)^2$ satisfies
\[
 t_k'(0)=0,\qquad t_k''(0)=2\bigl(x_k'(0)\bigr)^2\ne0.
\]
It therefore fails to be locally constant or flat at the zero,
but it is nonetheless a smooth plot of $[0,1]$ in the original
Magnot--Watts construction. This is exactly why the polynomial
squaring map $\Phi$ is smooth, whereas the positive square-root
section is not. The existence of a smooth square-root section
would require a \emph{different} diffeology on the barycentric
join, in which each admitted plot has smoothly varying positive
square roots (as well as a suitable local-support condition if
the target is our final spherical diffeology).

Requiring that every barycentric coordinate be identically zero
in a neighbourhood of \emph{each} of its zeros would suffice for
the local square-root calculation but would be far stronger than
stationarity of the homotopy parameter near $0$ and $1$. Indeed,
on a connected plot domain such a condition makes each zero set
both open and closed and hence forbids changes of support.
\end{remark}

\begin{theorem}[A smooth injection in the opposite direction]
\label{thm:power-injection}
There is a natural smooth injective map
\[
 J_2:E_{\rm LF}(G)\longrightarrow E^{\rm sph}(G),\qquad
 [(t_i,g_i)]\longmapsto
 \left[\left(\frac{t_i^2}{(\sum_jt_j^4)^{1/2}},g_i\right)\right].
\]
Its image is the nonnegative spherical sector
$E^{\rm sph}_+(G)=\{[(x_i,g_i)]:x_i\geq0\}$.
The inverse bijection from this sector to $E_{\rm LF}(G)$ is
not smooth in general. In particular $J_2$ is a smooth injection,
\emph{not} an isomorphism onto a diffeological subspace.
\end{theorem}
\begin{proof}
On any finite simplex $\Delta_I$ the denominator is a
strictly positive smooth function: at least one weight is nonzero.
The formula therefore defines a smooth map
$\Delta_I\times G^I\to S_I\times G^I$, compatible with
forgotten labels and with extension by zero. The universal
property of the final diffeology proves smoothness. For
$x_i\geq0$, the inverse formula at the set level is
\[
 t_i=\frac{\sqrt{x_i}}{\sum_j\sqrt{x_j}}.
\]
All sums are finite at each point. The formula proves
bijectivity onto the nonnegative sector. To test the inverse,
consider the smooth spherical plot
\[
 x_2(s)=s^2,\qquad x_1(s)=\sqrt{1-s^4}
\]
with fixed labels. The second component of the inverse is
$s\mapsto |s|/(\sqrt{\sqrt{1-s^4}}+|s|)$, which is not
differentiable at zero. Thus the inverse fails to be smooth,
already on the two-coordinate model.\end{proof}

\begin{remark}[Terminology]
The preceding theorem realizes a smooth injection with a
non-smooth inverse in the desired direction, but it concerns
the \emph{locally finite} variant $E_{\rm LF}(G)$ and uses a
power-normalization map, not square roots. It would be false
to identify it with the canonical positive-square-root section
of the original Magnot--Watts space. On the actual original
$E_{\rm MW}(G)$ there is instead the smooth surjection
$\Phi$ of Theorem~\ref{thm:comparison}.
\end{remark}

\subsection{Equivariant reindexing and smooth contractibility}
\label{subsec:contract-reindex}
The correct contraction must avoid superposing two unrelated
labels in an active coordinate. This can be done with a sequence
of \emph{elementary transfers to vacant indices}. The local finite
support built into the final diffeology is essential here.

Fix a smooth nondecreasing function $b:[0,1]\to[0,1]$
which equals zero near $0$ and one near $1$.
Let $\sigma:\mathbb N\to\mathbb N$ be strictly increasing
with $\sigma(n)\geq n$. Define the reindexing map
$R_\sigma:E^{\rm sph}(G)\to E^{\rm sph}(G)$ by sending
the entry $(x_n,g_n)$ to index $\sigma(n)$ and inserting
zeros elsewhere. This operation is smooth because on a
finite generating chart it is an ordinary finite-coordinate
inclusion (with permuted indices).

\begin{lemma}[Equivariant reindexing homotopy]
\label{lem:reindexing}
For every diffeological group $G$ and every such $\sigma$,
$R_\sigma$ is smoothly homotopic to the identity by a homotopy
which is equivariant for the commuting actions of $G$ and
$\mathbb Z_2$.
\end{lemma}
\begin{proof}
Divide $(0,1]$ into intervals
$I_n=[1/(n+1),1/n]$, $n\geq1$. For $t\in I_n$ put
$u_n(t)=(t-1/(n+1))/(1/n-1/(n+1))$ and
$\theta_n(t)=\frac\pi2 b(u_n(t))$.

We give an explicit formula for the homotopy $H_\sigma(z,t)$.
Represent $z$ by $(x_i,g_i)_{i\geq1}$. At time $t\in I_n$,
all original indices $k>n$ have already been transferred to
$\sigma(k)$; the indices $k<n$ have not been moved and still
occupy $k$. If $\sigma(n)>n$, transfer entry $n$ into the
vacant index $\sigma(n)$ by
\[
 \begin{array}{c|c|c}
 \text{new index}&\text{new spherical coordinate}&\text{label}\ \\ \hline
 n&x_n\cos\theta_n(t)&g_n\\
 \sigma(n)&x_n\sin\theta_n(t)&g_n.
 \end{array}
\]
If $\sigma(n)=n$, leave the entry unchanged. All other
coordinates stay at their described positions.
The target index $\sigma(n)>n$ is vacant: an original entry
$k<n$ still has index $k<n$, and a previously moved entry
$k>n$ occupies $\sigma(k)>\sigma(n)$.
Thus no different group labels are ever merged. If $x_n=0$,
both newly assigned coordinates vanish and the value is
independent of the forgotten label $g_n$. The sphere
normalization is preserved because
$\cos^2\theta_n+\sin^2\theta_n=1$.

At a common endpoint of $I_n$ and $I_{n-1}$, the formulas agree:
the transfer of $n$ has finished before the transfer of $n-1$
begins. The cut-off $b$ is constant near both endpoints, so
all derivatives glue there. Define $H_\sigma(z,0)=z$.
For a plot with local lift into $S_I\times G^I$, let
$N=\max I$. If $0<t<1/(N+1)$, every transfer concerns a
zero coordinate, so $H_\sigma(z,t)=z$ on a neighbourhood of
the plot parameter and $t=0$. This proves smoothness at the
accumulation point of the subdivision. Elsewhere there are
only finitely many relevant transfers locally in time.
For a fixed finite input support $I$, all intermediate outputs
are supported inside the finite set $I\cup\sigma(I)$; all
coordinate and label formulae are smooth in local lifts.
The final diffeology therefore proves that $H_\sigma$ is a
smooth map $E^{\rm sph}(G)\times[0,1]\to E^{\rm sph}(G)$.
At $t=1$ each index has reached $\sigma(n)$, giving $R_\sigma$.
The construction never alters labels and is linear in the
signed coordinates, so it commutes with both actions.\end{proof}

\begin{theorem}[Smooth contractibility for every diffeological group]
\label{thm:spherical-contract}
For every diffeological group $G$, the total spherical Milnor
space $E^{\rm sph}(G)$ with its finite-chart final diffeology
is diffeologically contractible.
\end{theorem}
\begin{proof}
Apply Lemma~\ref{lem:reindexing} to the shift
$\sigma(n)=n+1$. It yields a smooth homotopy from the
identity to the reindexed point $R_\sigma(z)$, whose first
spherical coordinate is zero. For $0\leq t\leq1$, define
$C_t$ on the image of the shift by
\[
 C_t\bigl([(0,(y_j,g_j)_{j\geq2})]\bigr)
 =\left[\left(\sin\frac{\pi t}2,e\right)_{j=1},
 \left(\cos\frac{\pi t}2\,y_j,g_j\right)_{j\geq2}\right].
\]
This is well-defined, even when one of its spherical
coordinates vanishes, because the attached group label is
then forgotten. Its sphere normalization is one and it is
smooth on every finite generating chart. At $t=0$ it is the
identity on the image of the shift; at $t=1$ it is the
constant point supported at coordinate one with label $e$.
Compose the shift homotopy with $C_t$ and concatenate them
using a smooth time reparametrization that is constant near
the junction. This gives a smooth contraction of the total
space. It is not $G$-equivariant, as cannot be expected for
a constant map on a nontrivial free $G$-space.\end{proof}

\begin{theorem}[Equivariant uniqueness of maps to the total space]
\label{thm:equivariant-homotopy}
Let $H$ denote either $G$, acting diagonally on the labels,
or $G\times\mathbb Z_2$ with the additional global sign action.
For any diffeological $H$-space $P$ and smooth equivariant maps
$f_0,f_1:P\to E^{\rm sph}(G)$, there exists a smooth
$H$-equivariant homotopy from $f_0$ to $f_1$.
\end{theorem}
\begin{proof}
Put $\sigma_e(n)=2n$ and $\sigma_o(n)=2n-1$.
Both are increasing, satisfy $\sigma(n)\geq n$, and have
disjoint images. Lemma~\ref{lem:reindexing} gives equivariant
homotopies $f_0\simeq R_{\sigma_e}f_0$ and
$f_1\simeq R_{\sigma_o}f_1$.
Writing the images of these two maps with their disjoint index
sets, define for $0\leq t\leq1$ a middle homotopy by
multiplying every spherical coordinate in the even image by
$\cos(\pi t/2)$, and every spherical coordinate in the odd
image by $\sin(\pi t/2)$, retaining their respective group
labels. The two supports are disjoint, so no collisions of
labels occur, and the squared norm is
$\cos^2(\pi t/2)+\sin^2(\pi t/2)=1$. On each local pair of
finite-chart lifts the formula is smooth, hence gives a
smooth map by the final diffeology. It is $H$-equivariant
because the $G$-action transforms all labels in the same way,
and the sign action multiplies all signed coordinates by $-1$.
This connects $R_{\sigma_e}f_0$ to $R_{\sigma_o}f_1$.
Concatenate with the two reindexing homotopies, reversing the
second and using time cut-offs constant near the junctions,
to obtain the required equivariant homotopy.\end{proof}

\subsection{Geometric interpretation}
The quadratic normalization provides smooth signed coordinates and
introduces genuinely new directions at loci with vanishing weights.
Its relationship with the classical Milnor model depends on the
chosen diffeology; it is not a diffeomorphism by construction.

\section{Diffeological smoothness and gluing}\label{4}
\subsection{Quotient charts and the filtered colimit}
For each finite nonempty $I$, let
\[
 Q_I=(S_I\times G^I)/\sim_I,
\]
where $(x,g)\sim_I(x,g')$ when $g_i=g'_i$ for all $i$ with
$x_i\ne0$, and give $Q_I$ the quotient diffeology. The chart map
$\psi_I:S_I\times G^I\to E^{\rm sph}(G)$ factors through
$Q_I$. For $I\subset J$, extension by zero gives a smooth map
$Q_I\to Q_J$.
\begin{proposition}
The spaces $Q_I$ form a filtered diagram and its colimit in
diffeological spaces is naturally diffeomorphic to $E^{\rm sph}(G)$.
\end{proposition}
\begin{proof}
The equivalence relation in each finite chart is precisely the
restriction of the defining relation on the total space. The
set-theoretic filtered colimit is therefore $E^{\rm sph}(G)$.
The final diffeology of the quotient charts agrees with the final
diffeology of the maps $S_I\times G^I\to E^{\rm sph}(G)$.
\end{proof}
\begin{remark}
It is \emph{not} correct to replace $Q_I$ by $S_I\times G^I$
in the colimit: different group labels at inactive indices must
be identified. Nor is the canonical map $S_I\times G^I\to Q_I$
injective. The images of the $Q_I$ form an increasing filtration,
not a disjoint stratification.
\end{remark}

\subsection{Plots and smooth functions}
\begin{proposition}
Every plot of $E^{\rm sph}(G)$ locally lifts to a smooth map into
some $S_I\times G^I$; in particular its spherical coordinates have
locally finite support. A map $f:E^{\rm sph}(G)\to Y$ to any
diffeological space $Y$ is smooth if and only if every composite
$f\circ\psi_I$ is smooth.
\end{proposition}
\begin{proof}
These statements are exactly the local-lifting property and the
universal property of the final diffeology.
\end{proof}
\subsection{Support strata}
The support of $p=[(x_i,g_i)]$ is the finite nonempty set
$\mathrm{supp}(p)=\{i:x_i\ne0\}$. Define
\[
 E_J^\circ=\{p\in E^{\rm sph}(G):\mathrm{supp}(p)=J\}.
\]
\begin{proposition}
For a finite nonempty set $J$, the support stratum $E_J^\circ$
with its subset diffeology is diffeomorphic to
\[
 S_J^\circ\times G^J,\qquad
 S_J^\circ=\{x\in S_J:x_j\ne0\text{ for every }j\in J\}.
\]
\end{proposition}
\begin{proof}
On $S_J^\circ$ there are no forgotten group coordinates. A plot
of the stratum locally lifts to a finite chart $S_I\times G^I$.
Near each point of the plot the active coordinates in $J$ stay
nonzero and all coordinates outside $J$ vanish identically,
since the image stays in the stratum. Restriction to $J$ gives
a smooth inverse to the indicated parametrization.
\end{proof}
\begin{remark}
The $E_J^\circ$ form a disjoint decomposition by supports. Their
dimensions may vary, and the inclusions of strata into the whole
space do not identify their tangent spaces with the full internal
tangent at a boundary point.
\end{remark}

\section{Tangent structures and boundary behavior}\label{5}
We now compute the internal tangent \emph{as a vector space} for
Fr\'echet--Lie groups. The diffeology on the resulting tangent
pseudo-bundle requires separate consideration. This distinction
is indispensable when constructing bilinear tensors.

\begin{definition}[Smooth separation of points]
\label{def:smooth-separation}
A Fr\'echet--Lie group $G$ is \emph{smoothly point-separating} if
for any two distinct $g,h\in G$ there is an $a\in C^\infty(G,\mathbb R)$
with $a(g)\ne a(h)$. This assumption is weaker than the
existence of smooth bump functions or smooth partitions of unity.
In particular, any finite collection of distinct points $g_1,\dots,g_m$
can be interpolated by smooth functions: for each $i\ne j$ choose
$f_{ij}$ with $f_{ij}(g_i)\ne f_{ij}(g_j)$ and set
\[
 a_i(g)=\prod_{j\ne i}
 \frac{f_{ij}(g)-f_{ij}(g_j)}{f_{ij}(g_i)-f_{ij}(g_j)}.
\]
Then $a_i(g_j)=\delta_{ij}$; no bump function is used.
\end{definition}

\begin{lemma}[Internal tangent of a Fr\'echet manifold]
\label{lem:frechet-internal-tangent}
Let $N$ be a smooth Fr\'echet manifold in the convenient
(equivalently, on Fr\'echet charts, Bastiani) calculus,
with its manifold diffeology. Its internal diffeological tangent
at $p$ is canonically isomorphic, as a real vector space,
to the usual Fr\'echet tangent space $T_pN$.
\end{lemma}
\begin{proof}
The ordinary derivative of every pointed plot defines a natural
linear map from the internal tangent to $T_pN$. It is onto: in
a chart with $p$ represented by $0$ every $v\in T_pN$
is realized by the smooth straight-line curve $t\mapsto tv$.
For injectivity, it is enough to treat pointed curves,
because the tangent vector represented by $(q,u)$ for a
pointed plot $q:U\to N$ is also represented by the curve
$t\mapsto q(tu)$ after translating coordinates on $U$.
In a Fr\'echet chart, let $c(0)=0$ and $c'(0)=v$. The
smooth Hadamard lemma for complete locally convex spaces,
applied twice to the smooth curve $c$, gives
$c(t)=tv+t^2r(t)$ for a smooth Fr\'echet-valued curve $r$.
The curve $t\mapsto t^2r(t)$ represents zero in the
internal tangent: it factors through the smooth map
$F(s,t)=st\,r(t)$, whose ordinary derivative at $(0,0)$
vanishes in both coordinate directions, and through the
diagonal $t\mapsto(t,t)$. Smoothness of addition and of
the coordinate chart identifies the tangent class of $c$
with that of $t\mapsto tv$. Thus a tangent class with zero
ordinary derivative is zero, proving injectivity.
\end{proof}

\subsection{Tangents along support strata and their retractions}
Assume throughout this section that $G$ is a Hausdorff,
smoothly point-separating Fr\'echet--Lie group. For $p=[(x_i,g_i)]$ set
$J=\operatorname{supp}(p)$. Recall
$E_J^\circ\simeq S_J^\circ\times G^J$.
\begin{lemma}[Retraction onto each stratum]
\label{lem:retract-support}
The inclusion $E_J^\circ\hookrightarrow E^{\rm sph}(G)$
induces an injection on internal tangent spaces at every
$p\in E_J^\circ$.
\end{lemma}
\begin{proof}
Consider the $D$-open set
$U_J=\{[(y_i,h_i)]:y_j\ne0\text{ for all }j\in J\}$.
Define
\[
 r_J([(y_i,h_i)])=\left[\left(
  \frac{y_j}{(\sum_{k\in J}y_k^2)^{1/2}},h_j
 \right)_{j\in J}\right]\in E_J^\circ.
\]
The positive denominator and active labels are intrinsic on
$U_J$, so $r_J$ is well-defined. On every generating chart it
is the ordinary smooth restriction-and-renormalization map;
therefore it is diffeologically smooth. On $E_J^\circ$ it is
the identity. Differentiating
$r_J\circ\iota_J=\operatorname{id}$ shows that
$(\iota_J)_*$ has a left inverse, hence is injective.\end{proof}

\begin{proposition}[Tangents to a fixed support stratum]
\label{prop:fixed-tangent}
There is a canonical identification
\[
 T_pE_J^\circ\simeq
 \left\{v\in\mathbb R^J:\sum_{j\in J}x_jv_j=0\right\}
 \oplus\bigoplus_{j\in J}T_{g_j}G.
\]
\end{proposition}
\begin{proof}
The inverse to $E_J^\circ\simeq S_J^\circ\times G^J$
is smooth by the lifting argument of Section~\ref{4}.
The first factor is an open submanifold of the sphere cut out
by $\sum_jx_j^2=1$, whose derivative is
$2\sum_jx_jv_j$. Lemma~\ref{lem:frechet-internal-tangent} identifies the internal tangent
with the ordinary Fr\'echet tangent of this finite product,
giving the displayed formula.\end{proof}

\subsection{Labelled activation branches}
For a set $A$, write $\mathbb R^{(A)}$ for the vector space
with basis $(\mathbf e_a)_{a\in A}$ and \emph{finite} linear
combinations. For $k\notin J$ and $h\in G$, choose a
representative of $p$ with the label at index $k$ equal to $h$.
The curve which activates coordinate $k$ with derivative one,
normalizing the remaining spherical coordinates, defines
an internal tangent vector $\mathbf e_{k,h}$ at $p$.
The construction is independent of all other inactive labels:
the corresponding coordinates remain zero along the curve.

\begin{lemma}[The labelled activation cone]\label{lem:activation-cone}
Let $G$ be a smoothly point-separating diffeological group, and
let
\[
 C(G)=(\mathbb R\times G)/(\{0\}\times G)
\]
carry the quotient diffeology, with vertex $o$. For $g\in G$ let
$c_g(s)=[s,g]$. Then the algebraic internal tangent is
\[
 T_o^{\mathrm{int}}C(G)\cong\mathbb R^{(G)},\qquad
 \mathbf e_g\longmapsto (c_g)_*(\partial_s|_0).
\]
\end{lemma}
\begin{proof}
A pointed plot of $C(G)$ lifts locally, by the quotient
property, to a plot $(r,\gamma)$ of $\mathbb R\times G$.
At a point over $o$, its differential in a direction $u$
represents $dr(u)\mathbf e_{\gamma(0)}$; the contribution from
$d\gamma(u)$ vanishes, since $g\mapsto[0,g]$ is constant.
This proves spanning. For $a\in C^\infty(G,\mathbb R)$, the
function $f_a([r,g])=r\,a(g)$ is well defined on equivalence
classes and smooth by the quotient diffeology. Its differential
satisfies $df_a(\mathbf e_g)=a(g)$. Given finitely many distinct
labels, smooth point separation supplies test functions with
arbitrary prescribed values on those labels (Definition~\ref{def:smooth-separation}).
Such differentials exclude all nontrivial finite relations
among the $\mathbf e_g$.
\end{proof}

\begin{lemma}[Finite-chart product normal form]\label{lem:finite-cone-normal}
Let $p\in Q_I$ have support $J\subseteq I$ and put $K=I\setminus J$.
In the $D$-open locus where all $x_j$, $j\in J$, are nonzero,
there is a smooth product identification
\[
 Q_I\supset U_J^I\ \cong\
 (S_J^\circ\times G^J)\times
 \left\{([r_k,h_k])_{k\in K}\in C(G)^K:
         \sum_{k\in K}r_k^2<1\right\}.
\]
Its inverse assigns to $(u,(g_j),([r_k,h_k]))$ the class of
\[
 x_j=\sqrt{1-\sum_{k\in K}r_k^2}\,u_j\quad(j\in J),
 \qquad x_k=r_k\quad(k\in K),
\]
with the indicated labels. Consequently,
\[
 T_pQ_I\cong T_pE_J^\circ\oplus
          \bigoplus_{k\in I\setminus J}\mathbb R^{(G)}.
\]
\end{lemma}
\begin{proof}
The inverse map is well defined on quotient classes: the
labels $h_k$ become irrelevant whenever $r_k=0$. Testing on
representatives shows that it is smooth. The forward map uses
$u_j=x_j/(\sum_{j\in J}x_j^2)^{1/2}$, the active $g_j$ and the
quotient classes $[x_k,g_k]$. Its pullback along the generating
subduction $S_I\times G^I\to Q_I$ is smooth, so it is smooth.
These constructions are mutually inverse. At $p$ the inactive
cone coordinates are vertices and their squared radii sum to
zero, so the inequality defines a neighbourhood of that point.
The internal tangent functor commutes with finite products
\cite{ChristensenWu2016}; Lemma~\ref{lem:activation-cone}
then gives the formula. The factor $S_J^\circ\times G^J$
may be infinite-dimensional when $G$ is Fr\'echet.
\end{proof}

\begin{lemma}[Finite-index retractions]\label{lem:finite-index-retract}
Let $I\supseteq J=\operatorname{supp}(p)$ be finite. The
canonical inclusion $\iota_I:Q_I\to E^{\mathrm{sph}}(G)$
induces an injection on internal tangents at $p$.
\end{lemma}
\begin{proof}
The $D$-open set
$W_I=\{[(x,g)]:\sum_{i\in I}x_i^2>0\}$
contains $p$. Define
\[
 R_I:W_I\longrightarrow Q_I,\qquad
 [(x,g)]\longmapsto
 \left[\left(\frac{x_i}{(\sum_{j\in I}x_j^2)^{1/2}},g_i\right)_{i\in I}\right].
\]
Missing coordinates are assigned zero weights, so their labels
are immaterial. Each composite with a generating chart is a
smooth restriction-and-renormalization map, hence $R_I$ is smooth.
Moreover $R_I\circ\iota_I=\operatorname{id}_{Q_I}$.
Differentiation at $p$ proves the claim.
\end{proof}

\begin{theorem}[Exact algebraic internal tangent]
\label{thm:internal-tangent-exact}
Let $G$ be a Hausdorff, smoothly point-separating
Fr\'echet--Lie group and let $p\in E^{\rm sph}(G)$
have finite support $J$. There is a
natural isomorphism of real vector spaces
\begin{equation}\label{eq:exact-internal-tangent}
 T_pE^{\rm sph}(G)\simeq T_pE_J^\circ\oplus
 \bigoplus_{k\notin J}\mathbb R^{(G)},
\end{equation}
where direct sums are algebraic. The first summand is injected
by the stratum inclusion, and $\mathbf e_{k,h}$ represents the
basis vector indexed by $h$ in the $k$-th summand.
In particular, for $G=\mathbb Z_2$ there are \emph{two}
independent first-order activation directions per inactive
index, rather than one.
\end{theorem}
\begin{proof}
Let $I\supseteq J$ range over finite sets. The local normal form
of Lemma~\ref{lem:finite-cone-normal} computes each $T_pQ_I$
as an active tangent factor and one independent labelled free
space for each $k\in I\setminus J$. Under inclusion $I\subseteq L$
the induced tangent map extends these coordinates by zero: the
newly introduced spherical weights vanish identically along
$Q_I$ and their labels do not contribute derivatives.
Lemma~\ref{lem:finite-index-retract} rules out collapse of any
of these directions in the full space.

Every pointed plot in $E^{\mathrm{sph}}(G)$ locally factors through
some finite quotient $Q_I$, so its internal tangent is generated
by the images of the finite-chart tangent spaces. These images
form an increasing filtered family and are injective by the
preceding lemma. Taking their algebraic union yields precisely
\[
 T_pE_J^\circ\oplus\bigoplus_{k\notin J}\mathbb R^{(G)}.
\]
In particular, this proof does not identify the tangent of the
quotient with a quotient of ordinary manifold tangent spaces,
which would incorrectly suppress the labelled branches.
\end{proof}

\begin{proposition}[Relative labels after the diagonal quotient]
\label{prop:relative-labels-quotient}
Fix $p\in E^{\mathrm{sph}}(G)$ and an active index $j$.
The $D$-open set $W_j=\{x_j\ne0\}$ admits a smooth slice for
the diagonal left $G$-action obtained by setting $g_j=e$.
Consequently, at $[p]\in E^{\mathrm{sph}}(G)/G$ there is an
algebraic splitting
\[
 T_{[p]}(E^{\mathrm{sph}}(G)/G)\cong
 T_{[p]}(E_J^\circ/G)\oplus
 \bigoplus_{k\notin J}\mathbb R^{(G)}.
\]
The labels of the last summands are relative labels $g_j^{-1}g_k$.
The further global antipodal quotient does not create new local
linear relations: one may additionally take the slice $x_j>0$.
\end{proposition}
\begin{proof}
The map $g_j:W_j\to G$ is smooth, since $j$ is active.
Writing $S_j=\{z\in W_j:g_j(z)=e\}$, the mutually inverse
maps
\[
 W_j\longleftrightarrow G\times S_j,\qquad
 z\longmapsto(g_j(z),g_j(z)^{-1}\cdot z),\qquad
 (a,s)\longmapsto a\cdot s
\]
are smooth, giving $W_j/G\cong S_j$. The finite-cone
normal form on this slice fixes one active label and leaves
exactly the displayed activation factors. Under the left
action the inactive labels transform to $g_j^{-1}g_k$.
Finally the two $D$-open sets $x_j>0$ and $x_j<0$ are
exchanged by the antipodal map and provide a smooth local
slice for that discrete action. No claim is made here for a
nontrivial semidirect action on the labels.
\end{proof}

\begin{remark}[What the theorem does and does not say]
The isomorphism \eqref{eq:exact-internal-tangent} is an
isomorphism of individual \emph{vector spaces}. It does not
assert that the family of free vector spaces
$\mathbb R^{(G)}$ assembles into a locally trivial bundle, or
that the displayed splitting is smooth with respect to the
Hector or dvs tangent-bundle diffeology. Those are distinct
questions. Smooth separation is the sole extra hypothesis needed in the
label-independence argument; it must not be replaced by the
unjustified assertion that arbitrary Fr\'echet manifolds admit
smooth bump functions.
\end{remark}

\subsection{Differentials of barycentric coordinates}
The coordinate projection
$\pi_x:E^{\rm sph}(G)\to S_{\rm fin}$ is smooth and induces
$(\pi_x)_*:T_pE^{\rm sph}(G)\to T_xS_{\rm fin}$.
Theorem~\ref{thm:internal-tangent-exact} shows that
$(\pi_x)_*$ maps each $\mathbf e_{k,h}$ to the same
coordinate vector $\partial_{x_k}$; consequently it forgets
the \emph{labelled} first-order branching.

\begin{proposition}[Quadratic vanishing]
\label{prop:quadratic-vanishing}
If $c(s)$ is a smooth spherical curve with $x_k(0)=0$ and
$t_k(s)=x_k(s)^2$, then
\[
 t_k'(0)=0,\qquad t_k''(0)=2(x_k'(0))^2.
\]
\end{proposition}
\begin{proof}
A curve locally lifts to $S_I\times G^I$ near zero; the
coordinate $x_k$ is independent of the chosen lift.
Apply the ordinary chain rule:
$t_k'=2x_kx_k'$ and
$t_k''=2(x_k')^2+2x_kx_k''$, then set $s=0$.\end{proof}

\begin{corollary}[The quadratic map kills first-order activations]
\label{cor:square-zero-normal}
Let $\Phi:E^{\mathrm{sph}}(G)\to E_{\mathrm{MW}}(G)$ be the
smooth squaring comparison of Section~\ref{3}. If
$k\notin\operatorname{supp}(p)$, then
\[
 (T_p\Phi)(\mathbf e_{k,h})=0\quad\text{for every }h\in G.
\]
The weight $t_k=x_k^2$ has a nontrivial second derivative
$2v_k^2$ along an activating spherical curve of velocity $v_k$;
this is a statement about second jets, not a canonical
identification with an iterated diffeological tangent space.
\end{corollary}
\begin{proof}
Choose the activation curve with $x_k(s)=s$ and
$x_j(s)=\sqrt{1-s^2}\,x_j(0)$ for $j\in J$, all active labels
fixed and the label $h$ fixed at $k$. After squaring, every
coordinate and label of its image depends evenly on $s$.
Thus $\Phi\circ c=\Phi\circ c\circ(-\operatorname{id}_{\mathbb R})$.
In the internal tangent colimit, these two parametrizations
represent opposite tangent vectors, which must therefore
vanish over $\mathbb R$. The second derivative is the elementary
calculation of Proposition~\ref{prop:quadratic-vanishing}.
\end{proof}

\begin{remark}[Second jets also retain relative label data]
For $a\in C^\infty(G,\mathbb R)$ and a fixed coordinate $k$,
the function $B_{k,a}([(t_i,g_i)])=t_k a(g_k)$ is smooth on
the original Magnot--Watts join: it is constant on the
forgotten-label equivalence classes, and its pullback to
the product-subset generating space is smooth. Along the
preceding curve with label $h$,
\[
 \left.\frac{d^2}{ds^2}\right|_{s=0}
 B_{k,a}(\Phi(c_h(s)))=2a(h).
\]
Thus a full second-order family of tests can detect labels,
although the single scalar weight $t_k''(0)=2$ cannot.
Neither statement implies a canonical isomorphism between
iterated tangents and second-jet spaces for these singular
diffeological quotients.
\end{remark}

\subsection{Implications for geometric tensors}
The tangent calculation identifies the missing data:
a smooth ambient metric must pair vectors labelled by distinct
$h\in G$ in the same inactive coordinate. The chartwise formula
of Section~\ref{6} specifies their individual squared lengths,
but not their mutual pairing. This is why the stratumwise metric
alone does not determine a global tangent metric. We return to this
point in Section~\ref{6}, where smooth positive kernels give
conditional extensions; compact finite-dimensional Lie groups
provide one example.

\section{Riemannian structure and energy formulation}\label{6}

The quotient construction of Sections~\ref{4}--\ref{5} calls for a
separation of three notions: a metric on a generating smooth chart,
a consistent energy assigned to locally lifted plots, and a smooth
bilinear metric on the full internal tangent pseudo-bundle. The first
two will be established below. The third requires extra data. The internal tangent has been
calculated in Theorem~\ref{thm:internal-tangent-exact}; an explicit
smooth tangent metric is constructed under a positive-kernel
hypothesis in Theorem~\ref{thm:kernel-ambient}.

\subsection{A weighted metric on generating charts}
Assume throughout this section that $G$ is a Hausdorff,
smoothly point-separating Fr\'echet--Lie group equipped with a
smooth positive-definite \emph{left-invariant} (possibly weak)
Riemannian metric $h$. In infinite dimension the existence of a
continuous positive inner product on the modeling space is an
explicit assumption, and ``Riemannian'' does not mean strong.
Left-invariance agrees with the action $g_i\mapsto ag_i$
from Section~\ref{3}. For every finite index set $I$, put
$P_I=S_I\times G^I$, with its product Fr\'echet-manifold structure.

\begin{definition}[Chart metric]\label{def:chartmetric}
For $(x,g)\in P_I$ and tangent vectors $(v,\xi)$ and $(w,\eta)$ set
\begin{align*}
 b_I{}_{(x,g)}((v,\xi),(w,\eta))
 &=\sum_{i\in I}v_iw_i+
 \sum_{i\in I}x_i^2 h_{g_i}(\xi_i,\eta_i),\\
 &\qquad\sum_{i\in I}x_iv_i=\sum_{i\in I}x_iw_i=0.
\end{align*}
Its associated energy is $\mathcal E_I(x,g;v,\xi)=b_I((v,\xi),(v,\xi))$.
\end{definition}

\begin{proposition}\label{prop:semi}
The field $b_I$ is a smooth positive semidefinite symmetric
$2$-tensor on $P_I$. At $(x,g)$ its radical is
\[
 \operatorname{rad}(b_I)_{(x,g)}=
 \bigoplus_{\{i\in I:x_i=0\}}T_{g_i}G.
\]
On $S_I^{\circ}\times G^I$, where all $x_i\ne0$,
it is a positive (possibly weak) Riemannian metric.
\end{proposition}
\begin{proof}
All coefficients are smooth. The Euclidean term is positive definite
in spherical tangent directions, while $h$ is positive definite and
$x_i^2>0$ exactly for active coordinates. This gives the claimed
radical and positive definiteness on the full-support locus.
\end{proof}

\begin{remark}\label{rem:radical}
The displayed radical is a subspace of $T_{(x,g)}P_I$, \emph{not} a
canonically identified subspace of $T_{[(x,g)]}Q_I$. In particular,
the old formula $\ker g_p=\bigoplus_{i\notin\operatorname{supp}(p)}T_{g_i}G$
for a metric on the quotient is not justified by Proposition~\ref{prop:semi}.
\end{remark}

\subsection{Compatibility with finite-coordinate inclusions}
For $J\subset I$, let $j_{JI}:P_J\to P_I$ insert zero spherical
coordinates and identity group elements outside $J$.
\begin{proposition}\label{prop:inclusionsmetric}
One has $j_{JI}^*b_I=b_J$ and
$\mathcal E_I\circ Tj_{JI}=\mathcal E_J$.
\end{proposition}
\begin{proof}
Both the added spherical derivatives and the added group derivatives
vanish along $j_{JI}$, leaving precisely the $J$-indexed terms.
\end{proof}

\subsection{An intrinsic energy on locally liftable plots}
Write $q_I:P_I\to Q_I$ for the finite quotient and
$\psi_I:Q_I\to E^{\mathrm{sph}}(G)$ for the canonical inclusion.
A plot of $E^{\mathrm{sph}}(G)$ locally has a lift
$\widetilde p=(x_i,g_i)_{i\in I}:U\to P_I$.
The smooth tensor $\widetilde p^*b_I$ need not originate from a
bilinear form on the internal tangent; nevertheless it is independent
of the lift in the following precise sense.

\begin{lemma}[Independence of representatives]\label{lem:liftindependent}
Let $\widetilde p=(x,g)$ and $\widetilde p'=(x',g')$ be smooth maps
$U\to P_I$ having the same image under $q_I$. Then
$\widetilde p^*b_I=(\widetilde p')^*b_I$ as smooth symmetric tensors
on $U$.
\end{lemma}
\begin{proof}
The quotient preserves the spherical coordinates, so $x=x'$. For
each $i$ the two group-coordinate maps agree on the open subset
$U_i=\{u:x_i(u)\ne0\}$. Hence their pulled-back $h$-tensors agree
there. On the interior of $\{x_i=0\}$, the corresponding weighted
terms vanish identically. These two open sets have dense union in
$U$: the omitted set is the boundary of the closed zero locus of
$x_i$, which has empty interior. Since the weighted terms are smooth,
their difference vanishes on all of $U$ by continuity. Summing over
$i$ proves the statement.
\end{proof}

\begin{proposition}[Consistent plotwise tensor]\label{prop:plotmetric}
For every plot $p:U\to E^{\mathrm{sph}}(G)$, local lifts define a
unique smooth symmetric $2$-tensor $b[p]$ on $U$ by
\[
 b[p]|_V=\widetilde p^{\,*}b_I
 \quad\text{whenever }p|_V=\psi_Iq_I\widetilde p.
\]
This tensor is independent of the chosen local lift and index set,
and for every smooth $f:V\to U$ it satisfies
$b[p\circ f]=f^*b[p]$.
\end{proposition}
\begin{proof}
Two lifts through different finite index sets can be embedded into
a common larger chart. Proposition~\ref{prop:inclusionsmetric}
and Lemma~\ref{lem:liftindependent} then give equality of the
resulting tensors. Consequently they glue on intersections of lift
domains. The final equality follows from functoriality of pullbacks.
\end{proof}

\begin{remark}[Scope of the construction]\label{rem:scopemetric}
Proposition~\ref{prop:plotmetric} constructs a \emph{plotwise quadratic
energy assignment}. The plotwise energy alone does not identify how vectors from
different labelled activation branches should be paired.
Theorem~\ref{thm:kernel-ambient} constructs such pairings once
an additional smooth positive kernel is chosen.
\end{remark}

\subsection{Stratumwise Riemannian geometry}
Recall the exact-support stratum $E_J^\circ$ of Section~\ref{4}.
It is naturally diffeomorphic to
$S_J^\circ\times G^J$, where all its spherical coordinates are nonzero.
\begin{theorem}[Metric on each support stratum]\label{thm:stratummetric}
The restriction of $b_J$ defines a smooth positive-definite
Riemannian metric $g_J^\circ$ on $E_J^\circ$, explicitly
\[
 g_J^\circ((v,\xi),(w,\eta))
 =\sum_{j\in J}v_jw_j+
 \sum_{j\in J}x_j^2h_{g_j}(\xi_j,\eta_j).
\]
The metrics $g_J^\circ$ are compatible with the plotwise energy of
Proposition~\ref{prop:plotmetric}.
\end{theorem}
\begin{proof}
On the exact-support stratum, the quotient forgets no active
coordinates and the standard product description applies. Every
coefficient $x_j^2$ is strictly positive, so
Proposition~\ref{prop:semi} yields the assertion. The compatibility
is immediate from the definition of $b[p]$.
\end{proof}

\begin{remark}
Although every $g_J^\circ$ is non-degenerate, no unique
ambient metric is forced by the collection. The conditional
metric of Theorem~\ref{thm:kernel-ambient} has many possible
choices and involves genuinely different labelled branches.
\end{remark}

\subsection{Equivariance}
The diagonal left action of $G$ on $P_I$ is
$a\cdot(x,(g_i))=(x,(ag_i))$; the global sign action is
$\varepsilon\cdot(x,g)=(\varepsilon x,g)$.
\begin{proposition}\label{prop:isometry}
Every $b_I$ is invariant under both these actions. The plotwise
assignment $b[p]$ and the metrics $g_J^\circ$ inherit the corresponding
invariance.
\end{proposition}
\begin{proof}
Left translations preserve $h$ by assumption, and do not change
$x_i$. The sign action preserves the Euclidean form and each
coefficient $x_i^2$. All constructions commute with the quotient
projections and inclusions.
\end{proof}
\begin{remark}
If the group action is kept as left multiplication, assuming only a
\emph{right-invariant} metric on $G$ does not suffice for the first
invariance claim. A bi-invariant metric is not needed: a
left-invariant one is enough.
\end{remark}

\subsection{Square-root coordinates and Fisher--Rao energy}
On a locus with a fixed nonempty support $J$, introduce the
barycentric coordinates $t_j=x_j^2$.
For variations tangent to a fixed sign chamber, $dt_j=2x_jdx_j$,
so
\[
 \sum_{j\in J}dx_j^2
 =\frac14\sum_{j\in J}\frac{dt_j^2}{t_j}.
\]
This is the usual Fisher--Rao quadratic form up to the conventional
factor $1/4$; the equality is an identity \emph{on the open support
stratum}, not a descent statement at $t_j=0$. In particular, the
map $x\mapsto (x_j^2)$ loses first-order information at inactive
indices, as observed in Section~\ref{5}.

\subsection{Limits of the construction and consequences}
The tensor $b_I$ has smoothly vanishing \emph{weights} at $x_i=0$;
this does not by itself construct a smooth Riemannian tensor on the
ambient internal tangent pseudo-bundle. Consequently, arguments
in subsequent sections that require a global musical isomorphism,
metric connection, Hodge adjoint, Clifford multiplication, or Dirac
operator cannot be deduced from this section alone. Their validity
requires either (i) restriction to support strata, or (ii) an
explicit additional tangent/pseudo-bundle model with a proven metric
and compatible connection.

\begin{remark}[Why an unweighted product metric does not descend]
An ordinary product metric on $S_I\times G^I$ assigns a positive
length to vectors tangent to a $G$-factor even when $x_i=0$.
Such vectors are killed by the quotient map $S_I\times G^I\to Q_I$.
Thus the unweighted product tensor cannot be the pullback of a
metric on the quotient. In the weighted tensor $b_I$, the term
$x_i^2h_{g_i}$ vanishes on exactly these forgotten-label
velocities. Nevertheless, the tensors $b_I$ alone leave
undetermined the pairings between \emph{different} activation
labels; this is supplied by $K$ below. The familiar quotient
Riemannian-submersion construction is therefore not being applied
to the unweighted product metric.
\end{remark}

\subsection{A smooth ambient metric from a positive kernel}
\label{subsec:kernelmetric}
Theorem~\ref{thm:internal-tangent-exact} makes it possible
to formulate a genuine, if non-canonical, metric on the
ambient internal tangent. A real kernel $K:G\times G\to\mathbb R$
will be called \emph{normalized strictly positive definite} if
$K(g,g)=1$, it is symmetric, and
$\sum_{a,b}c_ac_bK(g_a,g_b)>0$ for every finite collection of
distinct $g_a\in G$ and nonzero real vector $(c_a)$.
Assume in addition that $K$ is smooth and \emph{left-invariant}:
$K(ag,ah)=K(g,h)$.

\begin{remark}[Abstract groups versus smooth groups]
For any abstract group the Kronecker kernel
$K_\delta(g,h)=1$ if $g=h$ and $0$ otherwise is normalized,
strictly positive definite and left-invariant on the underlying
set. Its canonical feature map takes values in $\ell^2(G)$.
If $G$ is a non-discrete Lie group, however, this feature map
is discontinuous and $K_\delta$ is not smooth. Thus the
algebraic existence of a kernel does not meet the hypothesis
of the following theorem.
\end{remark}

\begin{theorem}[A conditional smooth internal tangent metric]
\label{thm:kernel-ambient}
Let $G$ be a Hausdorff, smoothly point-separating
Fr\'echet--Lie group equipped with a smooth positive-definite
left-invariant (possibly weak) metric $h$, and suppose it
admits a smooth left-invariant normalized strictly positive
definite kernel $K$. Then the algebraic tangent spaces of
Theorem~\ref{thm:internal-tangent-exact} carry smooth,
positive-definite bilinear forms $\mathsf g^K_p$ defined by
\[
 \begin{split}
 \mathsf g^K_p\bigl((u,\sum c_{k,g}\mathbf e_{k,g}),
                   (v,\sum d_{k,h}\mathbf e_{k,h})\bigr)
 =g_J^\circ(u,v)+\sum_{k\notin J}\sum_{g,h\in G}
        c_{k,g}d_{k,h}K(g,h).
 \end{split}
\]
They form a smooth symmetric bilinear function on the
internal tangent pseudo-bundle equipped with its dvs
diffeology, and pull back to the weighted chart tensors $b_I$
of Definition~\ref{def:chartmetric}. The diagonal left
$G$-action and the global sign action are isometries.
\end{theorem}
\begin{proof}
Every element of the decomposition in
Theorem~\ref{thm:internal-tangent-exact} is a \emph{finite}
linear combination of active stratum tangents and labelled
activation vectors. Therefore the displayed formula is a well-defined
bilinear pairing, and strict positive definiteness is immediate from
the positivity of $g_J^\circ$ and from the strictly positive Gram
matrices of $K$ for each inactive index. Invariance under the left
$G$-action follows from left invariance of $h$ and of $K$. The sign
action changes the signs of both members of every spherical pairing,
so also preserves the metric.

We prove smoothness for the \emph{dvs} tangent diffeology, rather
than silently identifying the algebraic splitting with a smooth
bundle decomposition. Recall that the dvs diffeology is generated
from Hector tangent plots by fibrewise smooth finite linear
combinations and the zero section
\cite[Definition~4.10 and Remark~4.7]{ChristensenWu2016}.
Let $q:U\to E$ be a plot, and let $V,W:U\to T^{\mathrm{dvs}}E$
be two tangent plots over $q$. After restriction to a neighbourhood
of an arbitrary $u_0\in U$ we can write
\[
 V=\sum_{\alpha=1}^{r}f_\alpha V_\alpha,\qquad
 W=\sum_{\beta=1}^{s}h_\beta W_\beta,
\]
where the coefficients are smooth and every $V_\alpha,W_\beta$
is obtained by the tangent of a plot of $E$, evaluated along
a smooth family of tangent vectors. Since the diffeology of $E$
is final with respect to finite generating charts, each such
plot has a local smooth lift into some $P_{I_\alpha}=S_{I_\alpha}\times
G^{I_\alpha}$. Differentiating the lift produces a smooth
family of ordinary Fr\'echet-manifold tangent vectors. Passing to
one finite common index set and inserting zero components and
constant identity labels where needed, it suffices to check
$u\mapsto\mathsf g^K(V_\alpha(u),W_\beta(u))$.

Represent the two lifted tangent families by
\[
 (x,g;v,\xi)\qquad\text{and}\qquad(x,h;w,\eta).
\]
Their base points represent the same $q(u)$, so their
spherical coordinates agree and $g_i=h_i$ wherever $x_i\ne0$.
Use left trivializations
\[
 \zeta_i=(dL_{g_i^{-1}})\xi_i,\qquad
 \theta_i=(dL_{h_i^{-1}})\eta_i.
\]
Then the pairing is
\begin{equation}\label{eq:kernelpairchart}
\mathsf g^K(V_\alpha,W_\beta)
 =\sum_i v_iw_i K(g_i,h_i)
 +\sum_i x_i^2\langle\zeta_i,\theta_i\rangle_{\mathfrak g}.
\end{equation}
Indeed, for $x_i\ne0$ both lifts use the same active label;
$K(g_i,h_i)=1$ and the two group tangent contributions give the
ordinary weighted $h$-pairing. For $x_i=0$, derivatives of
the inactive group labels map to zero in the quotient internal
tangent, while $v_i$ and $w_i$ determine the two activation
vectors with labels $g_i$ and $h_i$; their pairing is
$v_iw_iK(g_i,h_i)$. This proves the formula, including at
points where the support changes. Note that we do \emph{not}
separate the parameter domain into its zero and nonzero loci:
the same expression is smooth on the whole domain, because all
of $x,g,h,v,w,\zeta,\theta$ are smooth and $K,h$ are smooth.

The right hand side is unaffected by replacing a generator
with an equivalent pointed tangent representative: this follows
both from the explicit description of active and inactive tangent
classes in Theorem~\ref{thm:internal-tangent-exact}, and from the
fact that a label derivative occurring at an inactive index
has coefficient $x_i^2=0$. Thus the pointwise formula equals
the previously defined intrinsic bilinear pairing. We have
proved that the pairing of each pair of Hector generators is
smooth. By finite linear combinations,
\[
 \mathsf g^K(V,W)
 =\sum_{\alpha,\beta}f_\alpha h_\beta
      \mathsf g^K(V_\alpha,W_\beta)
\]
is smooth on $U$. This is precisely smoothness of
$\mathsf g^K:T^{\mathrm{dvs}}E\times_E T^{\mathrm{dvs}}E\to\mathbb R$.
Finally, when the two tangent vectors are obtained from the
same representative of $P_I$, one has $g_i=h_i$ throughout,
$K(g_i,g_i)=1$, and the pullback is exactly $b_I$.
\end{proof}

\begin{proposition}[Compact groups provide such kernels]
\label{prop:heatkernel}
Every compact finite-dimensional Lie group $G$ admits a kernel satisfying the
hypotheses of Theorem~\ref{thm:kernel-ambient}.
\end{proposition}
\begin{proof}
Choose a bi-invariant Riemannian metric and let $k_t$ denote its
heat kernel at time $t>0$, evaluated at the identity. The
function $K_t(g,h)=k_t(g^{-1}h)/k_t(e)$ is smooth, symmetric,
normalized, and left-invariant. Its Peter--Weyl Fourier
coefficients are strictly positive exponentials of the
nonnegative Casimir eigenvalues. For any finite signed measure
$\mu=\sum_ac_a\delta_{g_a}$, one has
\[
 \sum_{a,b}c_ac_b k_t(g_a^{-1}g_b)
   =\sum_{\pi\in\widehat G}d_\pi e^{-t\lambda_\pi}
                 \|\widehat\mu(\pi)\|_{\rm HS}^2\geq0.
\]
If this sum is zero, each matrix Fourier coefficient of $\mu$
vanishes. Matrix coefficients are dense in $C(G)$ by
Peter--Weyl; hence $\mu$ vanishes as a signed measure, forcing
all $c_a=0$ for distinct support points. Dividing by $k_t(e)>0$
proves strict positive definiteness. The standard heat-kernel
expansion converges smoothly for $t>0$, proving the asserted
regularity.\end{proof}

\subsection{Heat-kernel Hilbert realization and the circle example}
\label{subsec:hilbert-heat}
The preceding kernel pairing has a concrete geometric
interpretation when the kernel admits a smooth normalized Hilbert
feature map. The distinction matters: an arbitrary pair $(K,h)$
in Theorem~\ref{thm:kernel-ambient} need not come from the
\emph{same} Hilbert embedding.

\begin{proposition}[A metric induced by an equivariant feature map]
\label{prop:feature-metric}
Let $G$ be a Hausdorff, smoothly point-separating
Fr\'echet--Lie group and let
$\varphi:G\to\mathcal H$ be a smooth map into a real Hilbert
space, of constant unit norm, satisfying:
\begin{enumerate}
\item $\varphi(ag)=\rho(a)\varphi(g)$ for an orthogonal
representation $\rho$ of $G$;
\item each finite family $\{\varphi(g_a)\}$ with distinct
$g_a$ is linearly independent;
\item $d\varphi_g:T_gG\to\mathcal H$ is injective for every $g$.
\end{enumerate}
Then
\[
 \mathcal F([(x_i,g_i)])=(x_i\varphi(g_i))_{i\in\mathbb N}
 \quad\in\quad\ell^2(\mathbb N;\mathcal H)
\]
is a smooth, well-defined, $G$-equivariant map of unit norm.
Its Hilbert pullback defines a smooth positive-definite bilinear
form on $T^{\mathrm{dvs}}E^{\mathrm{sph}}(G)$. On $S_I\times G^I$
it pulls back to
\[
 \sum_{i\in I}dx_i^2+\sum_{i\in I}x_i^2\,
 \varphi^*g_{\mathcal H}^{(i)}.
\]
At one inactive index, the Gram pairing of activation branches
is $K(g,h)=\langle\varphi(g),\varphi(h)\rangle$.
\end{proposition}
\begin{proof}
The factor $x_i$ makes each coordinate independent of its
forgotten label at $x_i=0$. On a finite generating chart the
map is a finite sum of smooth maps, and hence it is smooth
for the final diffeology. Its squared Hilbert norm equals
$\sum_i x_i^2=1$. The tangent map to a Hilbert manifold is
smooth for the Hector and dvs tangent diffeologies; composing
its differential with the smooth Hilbert inner product gives
a smooth bilinear function on the dvs tangent pseudo-bundle.

At an active index $i$, the differential is
$v_i\varphi(g_i)+x_i d\varphi_{g_i}(\xi_i)$. Since
$\|\varphi\|=1$, the two terms are orthogonal. If this
component vanishes, then $v_i=0$ and, by injectivity of
$d\varphi$, $\xi_i=0$. At an inactive index it maps
$\sum_g c_g\mathbf e_{i,g}$ to
$\sum_g c_g\varphi(g)$, which vanishes only for zero
coefficients by finite linear independence. The exact tangent
splitting of Theorem~\ref{thm:internal-tangent-exact} therefore
shows that $d\mathcal F$ is injective in each fibre.
On a generating chart, differentiate $x_i\varphi(g_i)$:
$\langle\varphi(g_i),d\varphi_{g_i}\rangle=0$ eliminates the
mixed terms and gives the asserted weighted tensor. Equivariance
follows from the orthogonality of $\rho$. The antipodal
action changes $\mathcal F$ to $-\mathcal F$ and is also
isometric. A further quotient by the diagonal group requires
a separate horizontal metric construction; the equivariant
Hilbert map is not generally constant on $G$-orbits.
\end{proof}

\begin{corollary}[The heat kernel of a compact Lie group]
\label{cor:hilbert-compact}
Let $G$ be a compact finite-dimensional Lie group with a
bi-invariant Riemannian metric and heat kernel $p_t$, $t>0$.
With normalized Haar measure, the map
\[
 \varphi_t(g)(u)=\frac{p_{t/2}(g^{-1}u)}{\sqrt{p_t(e)}}
 \qquad (g,u\in G)
\]
is a smooth feature map $G\to L^2(G,\mathbb R)$ satisfying
all hypotheses of Proposition~\ref{prop:feature-metric}. It induces
$K_t(g,h)=p_t(g^{-1}h)/p_t(e)$ and a left-invariant
Riemannian metric $h_t=\varphi_t^*g_{L^2}$ on $G$.
\end{corollary}
\begin{proof}
Smooth dependence of the heat kernel on $(g,u)$ gives
Hilbert smoothness of $\varphi_t$. The heat semigroup identity
gives its unit norm and reproduces $K_t$. Left translations
act orthogonally on $L^2(G)$, so it is equivariant. Finite
linear independence is precisely the strict positive definiteness
proved in Proposition~\ref{prop:heatkernel}.
If $d\varphi_t|_g(\xi)=0$, the heat evolution of the nonzero
distributional derivative $\partial_\xi\delta_g$ would vanish.
The heat operator is injective on distributions on a compact
manifold (all Laplace eigenvalue multipliers $e^{-t\lambda}$
are nonzero), a contradiction. Thus $d\varphi_t$ is injective.
\end{proof}

\begin{example}[The compact group $U(1)$]
\label{ex:U1heat}
Write $z=e^{i\theta}$, put
$Z_t=\sum_{n\in\mathbb Z}e^{-tn^2}$ and regard the complex
Hilbert space $\ell^2(\mathbb Z;\mathbb C)$ as real.
The normalized equivariant feature map is
\[
 \varphi_t(z)=\frac{1}{\sqrt{Z_t}}
   \big(e^{-tn^2/2}z^n\big)_{n\in\mathbb Z}.
\]
Its kernel and induced group metric are
\[
 K_t(z,w)=\frac{1}{Z_t}\sum_{n\in\mathbb Z}
 e^{-tn^2}\operatorname{Re}(z^n\overline{w}^{\,n}),
 \qquad
 h_t=c_t\,d\theta^2,
 \quad c_t=\frac{\sum_n n^2e^{-tn^2}}{Z_t}>0.
\]
For distinct $z_1,\dots,z_m$ the vectors
$\varphi_t(z_a)$ are linearly independent: a finite relation
implies $\sum_a\lambda_a z_a^n=0$ for
$n=0,\dots,m-1$, and the Vandermonde matrix is invertible.
Thus one inactive spherical index contributes the free
space $\mathbb R^{(U(1))}$, whereas the active group tangent
is $\mathfrak u(1)\simeq\mathbb R$. The induced chart metric is
$\sum_i dx_i^2+c_t\sum_i x_i^2d\theta_i^2$.
\end{example}

\begin{remark}[The additional hypothesis in infinite dimension]
The heat-kernel realization above uses compactness and
finite-dimensional elliptic theory. For general Fr\'echet--Lie
groups neither the existence of a smooth normalized strictly
positive-definite invariant kernel nor the existence of a
smooth equivariant Hilbert immersion with finitely independent
orbit is asserted. Theorem~\ref{thm:kernel-ambient} remains
conditional in that category.
\end{remark}

\begin{remark}[Non-canonicity and the Laplacian obstruction]
Different times $t$ or kernels give different pairings between
activation branches, even though all recover the same chartwise
energy $b_I$. Thus $b_I$ does \emph{not} determine a unique
ambient metric. Moreover, neither a canonical measure on the
full diffeological space nor a compatible global Hodge adjoint
follows from Theorem~\ref{thm:kernel-ambient}.
Sections~\ref{sec:plotwise-Hausdorff}--\ref{sec:hodge-plotwise-obstruction}
identify the precise scope of plotwise volume constructions and
their limitations at support changes.
\end{remark}

\section{Differential forms on spherical Milnor spaces}\label{7}
\label{sec:spherical-forms}

The plotwise calculus of differential forms is intrinsic to a diffeological
space and does not require a tangent bundle, a metric, or a choice of
integration functional.  In particular, it remains meaningful at points
where the support of a spherical Milnor coordinate system changes.  This
section distinguishes that global calculus from the Riemannian calculus on
individual support strata.

\subsection{The diffeological de Rham complex}

\begin{definition}[Plotwise differential forms]\label{def:plotforms}
Let $X$ be a diffeological space.  A degree-$k$ differential form
$\omega\in\Omega^k(X)$ assigns an ordinary smooth $k$-form
$\omega_p\in\Omega^k(U)$ to each plot $p:U\to X$, subject to
\[
 \omega_{p\circ f}=f^*\omega_p
 \qquad\text{for every smooth } f:V\to U.
\]
The differential and exterior product are defined plotwise:
\[
 (d\omega)_p=d(\omega_p),\qquad
 (\omega\wedge\eta)_p=\omega_p\wedge\eta_p.
\]
\end{definition}

\begin{proposition}\label{prop:globaldeRham}
For every diffeological space $X$,
$\bigl(\Omega^\bullet(X),d,\wedge\bigr)$ is a differential graded
commutative algebra with $d^2=0$. Every smooth map $F:X\to Y$ induces
a morphism of differential graded algebras
$F^*:\Omega^\bullet(Y)\to\Omega^\bullet(X)$.
\end{proposition}
\begin{proof}
Compatibility with reparametrization follows from the naturality of
the usual exterior derivative and wedge product on Euclidean open
sets.  All identities can be checked on each plot, where they are the
ordinary de Rham identities.  The functoriality of pullback is
immediate from $\bigl(F^*\omega\bigr)_p=\omega_{F\circ p}$.
\end{proof}

\begin{remark}
No identification of $\Omega^1(X)$ with sections of a dual of the
internal tangent pseudo-bundle is needed here.  Such an identification
is not automatic for arbitrary diffeological spaces.
\end{remark}

\subsection{Descent of forms through the finite spherical quotients}

Retain the notation
\[
 P_I=S_I\times G^I,\qquad
 Q_I=P_I/\!\sim_I,\qquad
 \psi_I:P_I\longrightarrow E^{\mathrm{sph}}(G)
\]
from Sections~\ref{3}--\ref{4}; thus $\psi_I$ denotes the composite
of the finite quotient projection and its canonical map to the full
space.  We do not assume here that $G$ is finite-dimensional.

\begin{theorem}[Descent criterion for differential forms]
\label{thm:formdescent}
Restriction along the generating maps identifies
$\Omega^k(E^{\mathrm{sph}}(G))$ with the families
\[
  (\alpha_I)_{I\subset\mathbb N,\,0<|I|<\infty},
  \qquad \alpha_I\in\Omega^k(P_I),
\]
that satisfy the following \emph{plotwise descent condition}: whenever
$I,J$ are finite and
$u:U\to P_I$, $v:U\to P_J$ are smooth maps from the same
Euclidean open set such that $\psi_I\circ u=\psi_J\circ v$, then
\begin{equation}\label{eq:formsdescent}
 u^*\alpha_I=v^*\alpha_J.
\end{equation}
In particular, for $J\subset I$ and the zero-extension map
$j_{JI}:P_J\to P_I$, one has
$j_{JI}^*\alpha_I=\alpha_J$.  The families satisfying
\eqref{eq:formsdescent} are stable under $d$ and $\wedge$.
\end{theorem}
\begin{proof}
For $\omega\in\Omega^k(E^{\mathrm{sph}}(G))$ set
$\alpha_I=\psi_I^*\omega$. Naturality of a diffeological form gives
\eqref{eq:formsdescent}, and the formula for $j_{JI}$ follows from
$\psi_I\circ j_{JI}=\psi_J$.

Conversely, suppose a family satisfies \eqref{eq:formsdescent}.
A plot $p:U\to E^{\mathrm{sph}}(G)$ has, locally on $U$, a
factorization $p|_V=\psi_I\circ\widetilde p$ by a smooth map
$\widetilde p:V\to P_I$.  Define
$\omega_p|_V=\widetilde p^{\,*}\alpha_I$.
On overlaps, two such local factorizations produce the same
$k$-form by \eqref{eq:formsdescent}. Thus the ordinary forms glue
uniquely on $U$. The resulting assignment is compatible with smooth
reparametrizations and therefore defines $\omega\in\Omega^k(E^{\mathrm{sph}}(G))$.
Uniqueness follows from the local-lifting property.
Finally, $d$ and $\wedge$ preserve \eqref{eq:formsdescent} because
ordinary pullbacks commute with both operations.
\end{proof}

\begin{remark}[Why inclusion compatibility is insufficient]
The equality $j_{JI}^*\alpha_I=\alpha_J$ alone does not ensure
descent: two lifts with the same image may differ in their inactive
group coordinates, even within a single $P_I$.  Condition
\eqref{eq:formsdescent} is required for arbitrary pairs of such
lifts, not only for the zero-extension inclusions.
\end{remark}

\subsection{Elementary forms and the behavior of forgotten labels}

For $G$ a smooth manifold and $a\in C^\infty(G)$, fix an index $i$.
Define the function $F_{i,a}$ by
\[
 F_{i,a}([(x_j,g_j)])=
 \begin{cases}
 x_i^2 a(g_i),&x_i\ne0,\\
 0,&x_i=0.
 \end{cases}
\]
This is well-defined because the value at $x_i=0$ does not depend on
the forgotten coordinate $g_i$.

\begin{proposition}\label{prop:weightedfunctions}
The function $F_{i,a}$ is smooth on $E^{\mathrm{sph}}(G)$, and its
global differential is the plotwise form characterized on a generating
chart containing $i$ by
\[
 \psi_I^*(dF_{i,a})=2x_i a(g_i)\,dx_i+x_i^2\,da(g_i).
\]
Neither term should be interpreted as requiring the existence of a
globally defined vector field in the direction of a forgotten label.
\end{proposition}
\begin{proof}
On each $P_I$, with $i\in I$, the displayed expression for
$F_{i,a}\circ\psi_I$ is smooth.  If $i\notin I$, its pullback is zero.
These expressions define a smooth function by the final diffeology;
differentiation then gives the asserted formula.  The descent of the
resulting one-form is also a special case of
Theorem~\ref{thm:formdescent}.
\end{proof}

\subsection{Restrictions to support strata}

For finite nonempty $J\subset\mathbb N$, write
$\iota_J:E_J^\circ\hookrightarrow E^{\mathrm{sph}}(G)$.

\begin{proposition}\label{prop:formsrestriction}
Restriction induces a morphism of differential graded algebras
\[
 \iota_J^*:\Omega^\bullet(E^{\mathrm{sph}}(G))
   \longrightarrow\Omega^\bullet(E_J^\circ),
 \qquad \iota_J^*d=d\iota_J^*.
\]
If $G$ is a finite-dimensional Lie group, the target is the ordinary
de Rham complex of the smooth manifold
$E_J^\circ\simeq S_J^\circ\times G^J$.
\end{proposition}
\begin{proof}
The inclusion is smooth by the subset diffeology; the first assertion
follows from Proposition~\ref{prop:globaldeRham}.  The second follows
from the manifold identification in Section~\ref{4}.
\end{proof}

\begin{remark}
The restriction maps need not be surjective, and independently chosen
forms on different strata need not be the restrictions of a global
form.  Theorem~\ref{thm:formdescent} supplies the necessary
compatibility criterion.  In particular, the family of stratumwise
de Rham complexes and the global diffeological de Rham complex are
different objects.
\end{remark}

\subsection{Actions, quotients, and basic forms}

The commuting actions of $G$ and $\mathbb Z_2$ on
$E^{\mathrm{sph}}(G)$ induce pullbacks on its global de Rham complex.
Their invariant forms form differential graded subalgebras.
Let $\pi:X\to X/H$ denote any diffeological quotient by a smooth
group action (in particular, one of the quotients in Section~\ref{3}).
Then $\pi^*:\Omega^\bullet(X/H)\to\Omega^\bullet(X)$ is injective:
plots of $X/H$ locally lift through $\pi$. Its image is precisely the
forms satisfying the descent condition for pairs of local lifts of the
same quotient plot. Such forms may be called \emph{basic} in the
plotwise descent sense. Mere invariance under $H$ is not a sufficient
descent criterion in general.

\begin{remark}[No ambient horizontal complex has been defined]
The radical $\operatorname{rad}(b_I)$ in
Proposition~\ref{prop:semi} lies in $TP_I$ and cannot be
identified with a kernel inside the quotient tangent. Under the
kernel hypothesis, Theorem~\ref{thm:kernel-ambient} instead
constructs a positive metric on the latter.
Even when a horizontal distribution is independently specified, its
annihilator need not be preserved by $d$ without additional
invariance or involutivity hypotheses. We therefore use the
unrestricted global diffeological complex rather than postulating a
horizontal subcomplex.
\end{remark}

\subsection{Hausdorff volume densities along plots}
\label{sec:plotwise-Hausdorff}

The construction of \emph{Hausdorff diffeological volumes} in
\cite[Sections~4.3--4.4]{Magnotloops2019} associates geometric
measure data to individual plots of a Riemannian diffeological
space.  We revisit it here for the smooth positive ambient metric
of Theorem~\ref{thm:kernel-ambient}.  This makes explicit what is
available at a support change and what requires extra choices.

Suppose $G$ admits the smooth normalized equivariant Hilbert
realization of Proposition~\ref{prop:feature-metric}, and write
$X=E^{\mathrm{sph}}(G)$ and $g=\mathcal F^*\langle\,\cdot,\cdot\,\rangle$.
Let $p:U\subset\mathbb R^n\to X$ be a plot.  Its pullback
$g_p=p^*g$ is a smooth \emph{positive semidefinite} symmetric
$2$-tensor, even though $g$ is positive definite on every
algebraic internal tangent space: the map $T_up$ need not be
injective.  In standard coordinates define the metric Jacobian
\begin{equation}\label{eq:plot-Jacobian}
 J_g(p)(u)=\sqrt{\det\bigl(g_p(\partial_{u^a},\partial_{u^b})\bigr)_{a,b=1}^n}\,,
 \qquad d\mu_p=J_g(p)\,|du^1\cdots du^n|.
\end{equation}
The coefficient $J_g(p)$ is continuous and nonnegative, but
need not be smooth at points where $dp$ drops rank.  The associated
measure is the $n$-dimensional \emph{parametrized area measure},
counting multiplicity where appropriate; without injectivity or an
area formula it must not be identified with the pullback of an
intrinsic Hausdorff measure on $X$.

\begin{proposition}[Regular plotwise volume]
\label{prop:regular-plot-volume}
On the open subset $U_p^{\mathrm{reg}}\subset U$ where $g_p$ is
positive definite, $J_g(p)$ is smooth and positive.  On an
oriented component it defines the smooth top form
\begin{equation}\label{eq:plot-vol}
  \omega_p=J_g(p)\,du^1\wedge\cdots\wedge du^n.
\end{equation}
Under an orientation-preserving smooth reparametrization which is
a local diffeomorphism of the \emph{same} dimension, these forms
transform by pullback. The positive densities $d\mu_p$ enjoy the
analogous rule without an orientation.  No corresponding
pullback statement for volume forms is asserted when dimensions
differ or the differential loses rank.
\end{proposition}
\begin{proof}
The determinant is a strictly positive smooth function on the
regular locus. Its positive square root is smooth. The standard
transformation law of the metric matrix under an invertible
Jacobian gives the stated density formula, and the fixed
orientation gives the form version.
\end{proof}

\begin{example}[Failure of a smooth positive volume at a support change]
\label{ex:singular-plot-volume}
Consider the labelled cone
$Q=(\mathbb R\times U(1))/(\{0\}\times U(1))$ with its
induced metric from the feature map
$[r,z]\mapsto r\varphi_t(z)$, where the normalization of
Example~\ref{ex:U1heat} is understood. For a small interval of
angles let
\[
 p:(-\varepsilon,\varepsilon)\times(-\delta,\delta)\longrightarrow Q,
 \qquad p(r,\theta)=[r,e^{i\theta}].
\]
The pullback of the ambient metric is
\begin{equation}\label{eq:cone-pullmetric}
 g_p=dr^2+c_t r^2\,d\theta^2,\qquad
 J_g(p)=\sqrt{c_t}\,|r|.
\end{equation}
It is smooth as a symmetric tensor through $r=0$, but its
\emph{positive} area density has no smooth oriented form
representative $\sqrt{c_t}|r|\,dr\wedge d\theta$ across
$r=0$ in the fixed orientation.  On each of the regions $r>0$
and $r<0$ the usual Riemannian volume form exists.  The
local spherical model adds a smooth positive radial factor and
has the same $|r|$ obstruction.  In particular, a globally
nondegenerate internal-tangent metric does not imply smooth,
nondegenerate pullback metrics on all plots.
\end{example}

\subsection{Generalized complementary-degree forms}
\label{sec:complementary-forms}

The $\Omega^{\infty-q}$ terminology of
\cite[Section~4.4]{Magnotloops2019} uses the \emph{varying}
dimension of an oriented plot rather than an actual exterior
power of a fixed infinite-dimensional cotangent bundle.
More precisely, for $q\geq0$, a candidate complementary-form
family assigns to every oriented regular plot $p:U\subset\mathbb
R^n\to X$ with $n\geq q$ a smooth form
\[
  \alpha_p\in\Omega^{n-q}(U).
\]
One may impose compatibility under orientation-preserving
same-dimensional local diffeomorphisms of domains.  Forms of
this kind have degrees depending on $p$; in general they are
\emph{not} ordinary diffeological forms as in
Definition~\ref{def:plotforms}, whose degree is fixed and whose
naturality includes \emph{all} smooth reparametrizations.

\begin{definition}[Complementarity relation]
\label{def:complementarity-relation}
Let $\beta\in\Omega^q(X)$.  A compatible family
$(\alpha_p)_p$ of the above kind is \emph{complementary to
$\beta$ with respect to the plotwise volume} when
\begin{equation}\label{eq:complementarity-relation}
  \alpha_p\wedge p^*\beta=\omega_p
\end{equation}
for every oriented regular plot of dimension $n\geq q$.
We view \eqref{eq:complementarity-relation} as a \emph{relation}
between $\beta$ and a family of varying-degree forms, not as the
definition of a universally existing or unique linear operator.
The defining wedge equation is taken from
\cite[Section~4.4]{Magnotloops2019}; the regular-plot domain and
the limited reparametrization compatibility are additional
conventions made explicit here.  The original treatment sets the
candidate complementary form to zero when $q>n$ or the plotwise
volume vanishes.  This convention does not by itself yield a
smooth extension through a changing-rank plot, as
Example~\ref{ex:singular-plot-volume} shows.
\end{definition}

\begin{proposition}[Complementarity alone is not a Hodge star]
\label{prop:complementarity-not-star}
Even on an oriented Euclidean manifold, relation
\eqref{eq:complementarity-relation} can admit no solution or
several solutions.  It does not define a canonical map
$\Omega^q(X)\to\Omega^{\infty-q}(X)$.
\end{proposition}
\begin{proof}
On any regular oriented $n$-plot with nonzero volume, the
choice $\beta=0$ makes \eqref{eq:complementarity-relation}
impossible. For nonuniqueness let $X=\mathbb R^2$ with
$\omega=dx\wedge dy$ and $\beta=dx$.  All forms
\[
 \alpha=-dy+f(x,y)\,dx,\qquad f\in C^\infty(\mathbb R^2),
\]
satisfy $\alpha\wedge\beta=\omega$. The two failures persist
before any singularity of the diffeological space is considered.
\end{proof}

On a regular oriented $n$-plot, the metric selects an unambiguous
complement by the \emph{stronger} pointwise identity
\begin{equation}\label{eq:plot-Hodge-identity}
 \eta\wedge *_{g_p}(p^*\beta)
   =\langle\eta,p^*\beta\rangle_{g_p}\,\omega_p
   \qquad (\eta\in\Omega^q(U)),
\end{equation}
using the convention that the $q$-form $\eta$ is on the left.
Thus $\alpha_p=*_{g_p}(p^*\beta)$ is a well-defined smooth
$(n-q)$-form on $U_p^{\mathrm{reg}}$, and agrees with the
usual Hodge star on each smooth support stratum.  We now show
why this family cannot satisfy the full descent requirements
of the diffeological de Rham complex.

\section{Hodge-type geometry on support strata}\label{sec:hodge-strata}

In this section $G$ is a finite-dimensional Lie group equipped with
the left-invariant Riemannian metric $h$ of Section~\ref{6}.  For each
finite nonempty support $J$, the manifold
$M_J=E_J^\circ\simeq S_J^\circ\times G^J$ carries the smooth
Riemannian metric $g_J^\circ$ of
Theorem~\ref{thm:stratummetric}. All operators below are constructed
on $M_J$; their existence on the full diffeological space is
\emph{not} inferred from the plotwise energy.

\subsection{Obstructions to a global plotwise Hodge star}
\label{sec:hodge-plotwise-obstruction}

The preceding construction suggests a formal operation
$\Omega^q(X)\to\Omega^{\infty-q}(X)$, defined plot by plot.
The following two elementary obstructions distinguish that
notation from a Hodge operator on the singular total space.

\begin{theorem}[Varying dimension and labelled-branch obstruction]
\label{thm:obstruction-global-star}
Let $X=E^{\mathrm{sph}}(U(1))$ carry the Hilbert metric of
Example~\ref{ex:U1heat}. Then:
\begin{enumerate}
\item The rule
$\beta\mapsto\bigl(*_{g_p}(p^*\beta)\bigr)_p$
restricted to regular oriented plots cannot satisfy the ordinary
naturality identity of diffeological differential forms under
\emph{arbitrary} smooth reparametrizations of plots, even in a
smooth finite-dimensional Riemannian manifold.
\item At a point with one inactive coordinate, the internal
tangent has independent labelled directions
$\mathbf e_z,\mathbf e_w$ for distinct $z,w\in U(1)$, with
\[
 \|\mathbf e_z\wedge\mathbf e_w\|_g^2
 =K_t(z,z)K_t(w,w)-K_t(z,w)^2>0.
\]
Nevertheless the differential of every individual plot into the
local labelled cone $Q$ has rank at most one at its vertex.
Consequently, plotwise Jacobian volumes at the singular point
cannot, without additional tangent-exterior and compatibility
data, supply the values of a Hodge complement on all the
nonzero bivectors of the algebraic internal tangent.
\end{enumerate}
\end{theorem}
\begin{proof}
For (1), take the Euclidean plane with the identity $2$-plot
$p(x,y)=(x,y)$, and its restriction along
$f:\mathbb R\to\mathbb R^2$, $f(s)=(s,0)$.  For the constant
$0$-form $1$, one has
\[
 f^*(*_{\mathbb R^2}1)=f^*(dx\wedge dy)=0,
 \qquad *_{\mathbb R}(f^*1)=ds\ne0.
\]
The two expressions even have different degrees.  The same
dimension-changing obstruction occurs for local inclusions of
regular plots of $X$.

For (2), every pointed plot $p:U\to Q$ through the vertex
$o$ locally lifts, near the chosen parameter, to a smooth
$(r,z):U\to\mathbb R\times U(1)$ with $r=0$ at that
parameter.  Its differential has the form
\[
 (dp)_u(v)=dr_u(v)\,\mathbf e_{z(u)},
\]
so $\operatorname{rank}(dp)_u\leq1$ and
$\Lambda^2(dp)_u=0$. On the other hand, the feature map gives
$\langle\mathbf e_z,\mathbf e_w\rangle=K_t(z,w)$; strict
positive definiteness of the heat kernel yields the displayed
positive determinant.  Hence the wedge of the two distinct
activation directions is nonzero, although no single plot
through $o$ realizes it as the image of a tangent bivector.
The pointwise data of plotwise complementary volume forms
therefore do not by themselves prescribe a full exterior-algebra
Hodge dual on the internal tangent.  This does not rule out
constructions using additional global data.
\end{proof}

\begin{remark}[Scope of the obstruction]
The theorem does \emph{not} assert that every conceivable
infinite-dimensional Hodge-like operator is impossible.  It
excludes the automatic gluing of ordinary plotwise Hodge stars
under the standard diffeological pullback rule.  The
$\Omega^{\infty-q}$ families of
Section~\ref{sec:complementary-forms} may still be studied on
regular plots with their weaker, dimension-sensitive
compatibility rules.  A genuinely global operator would require,
in addition, a suitable space of completed exterior forms,
volume or determinant data, and a domain compatible with the
support-change singularities.  It need not coincide with the
stratumwise formal codifferentials below.
\end{remark}

\subsection{Metric on differential forms and integration}

The metric $g_J^\circ$ induces fiberwise inner products on
$\Lambda^kT^*M_J$, with the standard determinant convention
\[
 \langle\alpha_1\wedge\cdots\wedge\alpha_k,
 \beta_1\wedge\cdots\wedge\beta_k\rangle
 =\det\bigl(\langle\alpha_a,\beta_b\rangle\bigr).
\]
Write $d\mu_J$ for the Riemannian density.  Since this is a density,
no orientation is required. For compactly supported smooth forms
$\alpha,\beta$ define
\[
 (\alpha,\beta)_J=
 \int_{M_J}\langle\alpha,\beta\rangle_{g_J^\circ}\,d\mu_J.
\]
This inner product is defined independently of a Hodge star.

\subsection{Codifferential and Laplacian}

\begin{theorem}[Stratumwise Hodge complex]\label{thm:stratumhodge}
For each finite nonempty $J$, the exterior derivative has a
well-defined formal $L^2$ adjoint
\[
 \delta_J:\Omega^{k+1}(M_J)\longrightarrow\Omega^k(M_J),
\]
characterized for compactly supported forms by
\[
 (d\alpha,\beta)_J=(\alpha,\delta_J\beta)_J.
\]
Writing $\nabla^J$ for the Levi--Civita connection and
$(e_a)$ for any local orthonormal frame, one has
\[
 \delta_J\beta=-\sum_a\iota_{e_a}\nabla^J_{e_a}\beta.
\]
The Hodge Laplacian
$\Delta_J=d\delta_J+\delta_Jd$ is a second-order elliptic
differential operator, formally self-adjoint and nonnegative:
\[
 (\Delta_J\omega,\omega)_J
  =\|d\omega\|_{L^2}^2+\|\delta_J\omega\|_{L^2}^2
 \quad (\omega\in\Omega_c^\bullet(M_J)).
\]
\end{theorem}
\begin{proof}
Fix an arbitrary point and choose a local oriented orthonormal
frame $(e_a)$ normal at that point.  Torsion-freeness gives the
pointwise identity $d=\sum_ae_a^\flat\wedge\nabla^J_{e_a}$.
For $\alpha\in\Omega_c^k(M_J)$ and
$\beta\in\Omega_c^{k+1}(M_J)$, one has
\[
 \langle e_a^\flat\wedge\nabla^J_{e_a}\alpha,\beta\rangle
  =\langle\nabla^J_{e_a}\alpha,\iota_{e_a}\beta\rangle.
\]
Integration of the divergence of the compactly supported vector
field whose components are
$\langle\alpha,\iota_{e_a}\beta\rangle$
gives the formal adjoint
$\delta_J=-\sum_a\iota_{e_a}\nabla^J_{e_a}$ at that point.
Both sides are coordinate-invariant, so the identity holds
globally. For compactly supported test forms,
$\langle d^2\alpha,\beta\rangle=0$ implies
$\langle\alpha,\delta_J^2\beta\rangle=0$; the arbitrariness
of $\alpha$ shows $\delta_J^2=0$ as a differential operator.
The formal adjoint identities then imply
\[
 (\Delta_J\omega,\omega)_J
   =(\delta_Jd\omega,\omega)_J+(d\delta_J\omega,\omega)_J
   =\|d\omega\|^2+\|\delta_J\omega\|^2.
\]
Finally the symbols at a covector $\xi$ are
$\sigma(d)(\xi)=\xi\wedge$ and
$\sigma(\delta_J)(\xi)=-\iota_{\xi^\sharp}$, omitting
common factors of $i$. Their anticommutator equals
$-|\xi|^2\operatorname{id}$; with the standard differential
operator symbol convention the resulting positive Laplacian
has principal symbol $|\xi|^2\operatorname{id}$.
Thus it is elliptic and formally self-adjoint.
\end{proof}

\begin{remark}[Formal versus Hilbert-space adjoints]
Theorem~\ref{thm:stratumhodge} is a statement about differential
operators and compactly supported test forms.  Essential
self-adjointness and the identification of a specific unbounded
$L^2$ realization require a choice of domain and, on the generally
incomplete manifold $M_J$, possibly boundary conditions.  These
analytic questions are distinct from the existence of the formal
codifferential.
\end{remark}

\subsection{Warped coefficients and failure of automatic gluing}

On a fixed sign chamber of $S_J^\circ$, the metric reads
\[
 g_J^\circ=g_{S_J}+\sum_{j\in J}x_j^2h^{(j)}.
\]
For an orthonormal frame $f_{j,a}$ of $(G,h)$, the vertical vectors
$|x_j|^{-1}f_{j,a}$ are unit vectors on the chamber.  The second-order
vertical part of the scalar Laplace--Beltrami operator consequently
contains terms $x_j^{-2}\Delta_{G,j}$.  These coefficients are
well-defined on a support stratum, but may diverge as a coordinate
approaches zero.  Such divergence is not, by itself, cancelled by
the vanishing of the coefficient $x_j^2$ in the metric.

\begin{example}[A smooth quotient function whose Laplacian does not descend]
\label{ex:laplace-obstruction}
Take $G=S^1$ with its standard metric, and restrict to the finite
model with $I=\{1,2\}$.  Near the locus $x_2=0$, write
\[
 x_1=\cos s,\quad x_2=\sin s,
 \qquad |s|<\frac\pi2.
\]
With angular group coordinates $\theta_1,\theta_2$, the full-support
metric on $s\ne0$ is
\[
 g=ds^2+\cos^2s\,d\theta_1^2+\sin^2s\,d\theta_2^2.
\]
The function
\[
 F=\sin^2s\cos\theta_2
\]
is the restriction of the globally smooth function $F_{2,a}$ of
Proposition~\ref{prop:weightedfunctions} for $a(e^{i\theta})=
\cos\theta$. It has value $0$ at $s=0$, regardless of the
forgotten label $\theta_2$.  For the scalar Laplace--Beltrami
operator $\Delta_{\mathrm{LB}}=\operatorname{div}\operatorname{grad}$
on the full-support stratum, direct calculation gives
\[
 \begin{split}
 \Delta_{\mathrm{LB}}F
 &=\frac1{\sin s\cos s}
   \partial_s\bigl(\sin s\cos s\,\partial_sF\bigr)
   +\frac1{\sin^2s}\partial_{\theta_2}^2F\\
 &=(4\cos 2s-1)\cos\theta_2.
 \end{split}
\]
As $s\to0$, this approaches $3\cos\theta_2$, which depends on
the group coordinate that is forgotten at $s=0$.  Hence the
stratumwise Laplacian of a globally smooth function need not extend
even continuously to the finite quotient.  The positive Hodge
Laplacian on functions has the opposite sign and the same
obstruction.
\end{example}

\begin{corollary}\label{cor:no-global-laplacian}
For $G=S^1$, there is no operator
$\Delta:\Omega^0(E^{\mathrm{sph}}(G))\to
\Omega^0(E^{\mathrm{sph}}(G))$ whose restriction to every full-support
stratum coincides, on every globally smooth function, with the
scalar Hodge Laplacian associated with $g_J^\circ$.
\end{corollary}
\begin{proof}
Were such an operator to exist, its application to the globally
smooth $F_{2,a}$ of Example~\ref{ex:laplace-obstruction} would
have to give a smooth quotient function.  The restrictions along
smooth curves approaching the same point with different fixed
values of $\theta_2$ have different limits, contradicting
continuity of the pullback along those curves.
\end{proof}

\begin{remark}[Possible repairs]
One may restrict to selected classes of forms or functions with
prescribed asymptotics, impose boundary or transmission conditions,
or change the operator by lower-order corrections.  Such choices
constitute additional analytic data. They cannot be recovered merely
from the pointwise degeneration of $b_I$, and are not made in this
article.
\end{remark}

\subsection{Harmonicity, equivariance, and analytic extensions}

A form $\omega\in\Omega^k(M_J)$ is called \emph{stratumwise
harmonic} if $\Delta_J\omega=0$.  This definition is independent of
a realization of the Laplacian as an unbounded operator, but no global
Hodge decomposition is asserted.  The diagonal $G$-action and the
global sign action are isometries of the stratum, so pullback by
these actions commutes with $d$, $\delta_J$, and $\Delta_J$.

If $G$ is infinite-dimensional, the finite size of $J$ does not make
$G^J$ finite-dimensional.  A formal Hodge construction then requires
independent choices of metric completion, integration functional,
operator domains, and suitable closed extensions. No such structure
is obtained automatically from Section~\ref{6}.

\section{Clifford structures and Dirac operators on support strata}
\label{8}\label{sec:dirac-strata}

We retain the finite-dimensional hypotheses on $G$ from the previous
section.  The Riemannian structure used below is the metric
$g_J^\circ$ on each smooth support stratum $M_J$, rather than the additional, non-canonical ambient metric
of Theorem~\ref{thm:kernel-ambient}.

\subsection{Canonical Clifford module and the de Rham--Dirac operator}

\begin{definition}[Canonical Clifford action]
On the exterior-algebra bundle
$\mathcal E_J=\Lambda^\bullet T^*M_J$, define
\[
 c_J(v)\omega=v^\flat\wedge\omega-\iota_v\omega,
 \qquad v\in TM_J.
\]
Here $v^\flat=g_J^\circ(v,\cdot)$.
\end{definition}

\begin{proposition}\label{prop:cliffordrelation}
For tangent vectors $v,w$ on $M_J$,
\[
 c_J(v)c_J(w)+c_J(w)c_J(v)
  =-2g_J^\circ(v,w)\operatorname{Id}.
\]
Thus $\mathcal E_J$ is a smooth module for the Clifford algebra of
$(TM_J,g_J^\circ)$.
\end{proposition}
\begin{proof}
Write $\varepsilon(v)=v^\flat\wedge$ and
$\iota(v)=\iota_v$. Then the graded anti-commutation rules are
\[
 \varepsilon(v)\varepsilon(w)+\varepsilon(w)\varepsilon(v)=0,
 \qquad \iota(v)\iota(w)+\iota(w)\iota(v)=0,
\]
and $\iota(v)\varepsilon(w)+\varepsilon(w)\iota(v)
=g_J^\circ(v,w)\operatorname{id}$. Expanding
$(\varepsilon(v)-\iota(v))(\varepsilon(w)-\iota(w))$
and adding the expression with $v,w$ interchanged cancels the
pure exterior and pure contraction terms. The two mixed pairs
each contribute $-g_J^\circ(v,w)\operatorname{id}$,
which proves the identity. Smoothness follows because wedge,
contraction, and the metric are smooth bundle operations.
\end{proof}

\begin{theorem}[Canonical de Rham--Dirac operator]
\label{thm:deRhamDirac}
Let $\nabla^J$ be the Levi--Civita connection of $g_J^\circ$
on differential forms.  The first-order operator
\[
 D_J=\sum_a c_J(e_a)\nabla^J_{e_a}
\]
where $(e_a)$ is a local orthonormal frame is independent of that
frame and satisfies
\[
 \boxed{D_J=d+\delta_J,\qquad D_J^2=\Delta_J.}
\]
Its principal symbol is Clifford multiplication by the dual of the
covector (up to the conventional factor $i$), so $D_J$ is elliptic.
The construction requires no choice of spin structure.
\end{theorem}
\begin{proof}
The expression $\sum_a c_J(e_a)\nabla^J_{e_a}$ is the
Clifford contraction of the covariant derivative and is invariant
under changes of orthonormal frame.  Since the connection is
torsion-free, the exterior product part is
$\sum_a e_a^\flat\wedge\nabla^J_{e_a}=d$; by
Theorem~\ref{thm:stratumhodge}, the contraction part is
$-\sum_a\iota_{e_a}\nabla^J_{e_a}=\delta_J$.
The square is therefore $d\delta_J+\delta_Jd$, since
$d^2=\delta_J^2=0$ on ordinary differential forms.
For a nonzero covector $\xi$, the Clifford relation yields
$c_J(\xi^\sharp)^2=-|\xi|^2\operatorname{Id}$; thus the principal
symbol is invertible away from the zero section, proving
ellipticity.
\end{proof}

\subsection{Other Clifford modules and compatible connections}

\begin{proposition}[Conditional Dirac-type construction]
\label{prop:conditional-clifford}
Let $\mathcal W_J\to M_J$ be a smooth vector bundle equipped with a
Clifford action $c$ of $\mathrm{Cl}(TM_J,g_J^\circ)$ and a
Clifford-compatible connection $\nabla^{\mathcal W_J}$, i.e.
\[
 \nabla^{\mathcal W_J}_X(c(v)s)
 =c(\nabla^J_Xv)s+c(v)\nabla^{\mathcal W_J}_Xs.
\]
Then
\[
 D_{\mathcal W_J}=\sum_a c(e_a)\nabla^{\mathcal W_J}_{e_a}
\]
is a well-defined first-order elliptic operator with principal
symbol $c(\xi^\sharp)$ (up to the factor $i$).  If the Clifford
module and its connection are Hermitian-compatible in the standard
sense, the operator is formally symmetric on compactly supported
sections. Its square has a Weitzenb\"ock decomposition
\[
 D_{\mathcal W_J}^2
  =(\nabla^{\mathcal W_J})^*\nabla^{\mathcal W_J}
    +\mathcal R_{\mathcal W_J}
\]
for the curvature endomorphism determined by the chosen module and
connection.
\end{proposition}
\begin{proof}
Let $A=(a_{ij})$ be an orthogonal change of local frame.
Then $e'_i=\sum_ja_{ij}e_j$ and
$\sum_i c(e'_i)\nabla_{e'_i}
=\sum_{i,j,k}a_{ij}a_{ik}c(e_j)\nabla_{e_k}
=\sum_j c(e_j)\nabla_{e_j}$, proving frame independence.
The first-order symbol is $\sigma(D)(\xi)=c(\xi^\sharp)$
(up to $i$); its square equals $-|\xi|^2\operatorname{id}$,
so it is invertible for $\xi\ne0$. In a normal orthonormal
frame at a given point, compatibility of the Hermitian metric
and the skew-adjoint Clifford matrices, followed by integration
by parts, gives formal symmetry on compactly supported sections.
Expanding the square at that point and separating the symmetric
and antisymmetric products of Clifford matrices gives
\[
 D_{\mathcal W_J}^2=(\nabla^{\mathcal W_J})^*\nabla^{\mathcal W_J}
  +\frac12\sum_{a,b}c(e_a)c(e_b)
           R^{\mathcal W_J}(e_a,e_b),
\]
where the sign of $R$ is defined by
$R(X,Y)=[\nabla_X,\nabla_Y]-\nabla_{[X,Y]}$.
This is the asserted Weitzenb\"ock decomposition, and the
curvature term is a zeroth-order bundle endomorphism.
\end{proof}

\begin{remark}
A Clifford action by itself does not select a unique compatible
connection on an arbitrary module.  In particular, compatibility of
transition maps for modules is not enough to infer that arbitrary
locally chosen connections glue.  This is why the connection is
included among the hypotheses.
\end{remark}

\subsection{Existence of spin structures on the strata}

\begin{proposition}\label{prop:spin-strata}
Every connected component of $M_J=E_J^\circ$ is parallelizable,
and hence admits spin structures.
\end{proposition}
\begin{proof}
Fix the signs of all $x_j$, $j\in J$. The corresponding component
of $S_J^\circ$ is diffeomorphic, via $t_j=x_j^2$, to the interior
of the simplex $\Delta_J$, hence to $\mathbb R^{|J|-1}$.
The tangent bundle of this component is trivial.  Every
finite-dimensional Lie group is parallelizable by left
translations, and so is the finite product $G^J$ (componentwise).
Their product therefore has trivial tangent bundle.  An oriented
trivialization produces a spin structure.
\end{proof}

\begin{remark}
Existence does not make the spin structure, spinor module, or
spinorial connection canonical.  The spin Dirac operator can be
constructed on any chosen spinor bundle with its induced connection;
there is no reason to identify spinors on distinct support strata
without extra transition data.
\end{remark}

\subsection{Equivariance and behavior at support changes}

\begin{proposition}\label{prop:dirac-equivariance}
The operators $D_J=d+\delta_J$ commute with pullbacks induced by
the diagonal left $G$-action and by the global $\mathbb Z_2$-action
on $M_J$.
\end{proposition}
\begin{proof}
Both actions preserve $g_J^\circ$ by
Proposition~\ref{prop:isometry}. Pullback by a Riemannian isometry
commutes with $d$, the formal adjoint $\delta_J$, and thus $D_J$.
\end{proof}

\begin{remark}[The support-change obstruction]
An identity such as
\[
 D_{\mathcal W_J}\sim D_{S_J}+\sum_{j\in J}|x_j|^{-1}D_{G,j}
\]
can at most describe the \emph{leading first-order behavior} on a
fixed sign chamber, once a particular Clifford module and connection
have been selected.  The precise lower-order terms depend on
connection and warping data, and the expression is not an operator
on the quotient at $x_j=0$.  Neither the vanishing of $x_j^2$ in
the metric nor the removal of an inactive group label establishes
cancellation of the $|x_j|^{-1}$ coefficients.  In fact,
Corollary~\ref{cor:no-global-laplacian}, together with
$D_J^2=\Delta_J$, rules out a universal globally descending
de Rham--Dirac operator that simultaneously satisfies these identities
on all globally smooth forms.
\end{remark}

\subsection{Infinite-dimensional groups and genuinely global operators}

For infinite-dimensional $G$, the stratum $M_J$ may already be
infinite-dimensional.  To define a Dirac operator there one must
independently supply a suitable metric tangent model, a Clifford
representation, a connection, and an analytically convergent
Clifford contraction on a specified domain.  Frame invariance and
closability are additional issues, not formal consequences of
having only finitely many Milnor indices.

Likewise, a Dirac operator on the \emph{full} diffeological spherical
Milnor space would require a compatible tangent or cotangent
pseudo-bundle model, a metric and Clifford module extending the
stratumwise data, and explicit compatibility at support changes.
This has not been constructed in the present paper.  The global
object established in Section~\ref{sec:spherical-forms} is instead
the plotwise de Rham complex; the Hodge and Dirac operators above
are genuinely stratumwise.

\section{Local systems, twisted forms, and higher-geometric data}
\label{sec:twisted-revised}

In this section $X=E^{\mathrm{sph}}(G)$, with the diffeology of
Section~\ref{3}.  The global de Rham complex of Section~\ref{7}
requires no metric.  By contrast, metric adjoints and Dirac operators
remain confined to support strata, subject to the hypotheses of
Sections~\ref{sec:hodge-strata}--\ref{sec:dirac-strata}.

\subsection{Flat sign systems and their cohomological degree}
A \emph{smooth sign local system} $L$ will mean a locally trivial
rank-one real diffeological vector bundle admitting local parallel
frames whose transition functions are locally constant and belong to
$\{\pm1\}$.  For example, it can be specified by a smooth \v{C}ech
$1$-cocycle $\varepsilon_{ab}:U_a\cap U_b\to\{\pm1\}$ on a
suitable diffeologically open covering; this is additional data,
not a consequence of the existence of the global sign involution.
When a comparison with ordinary cohomology is available, the class
of $L$ lies in $H^1(X;\mathbb Z_2)$.  For an arbitrary diffeological
space, a bare singular cohomology class does not, by itself, specify
the required smooth bundle and local trivialisations.

\begin{proposition}[Twisted de Rham differential]
\label{prop:twisted-d}
A smooth sign local system $L$ has a canonical flat connection
$\nabla^L$ determined by its locally constant transition functions.
It defines a natural differential
\[
 d_L:\Omega^k(X;L)\longrightarrow\Omega^{k+1}(X;L),
 \qquad d_L^2=0.
\]
\end{proposition}
\begin{proof}
In a local parallel frame, $L$-valued forms are ordinary plotwise
forms and $d_L$ is the exterior derivative on each plot.  On an
overlap the frames differ by a locally constant sign, with which
exterior differentiation commutes.  The local differentials
therefore agree and square to zero.
\end{proof}

\begin{remark}[The global sign action]
The smooth involution $x\mapsto-x$ gives a different source of
local systems: the map
$E^{\mathrm{sph}}(G)/G\to E^{\mathrm{sph}}(G)/(G\times\mathbb Z_2)$
is a free $\mathbb Z_2$-quotient.  The associated sign line, when
formed over this \emph{projective quotient}, is not automatically a
line bundle on the original total space $X$.  These two bases must
not be confused.
\end{remark}

\subsection{Twisted Hodge and Clifford operators on strata}
Suppose $G$ is finite dimensional and let $M_J=E_J^\circ$.
The restriction $L|_{M_J}$ is a flat Euclidean line bundle.
Using the metric $g_J^\circ$ of Section~\ref{6}, define
\[
 D_{J,L}=d_L+\delta_{J,L}
 \quad\text{on }\Omega^\bullet(M_J;L|_{M_J}),
\]
where $\delta_{J,L}$ is the formal adjoint of $d_L$ for the
Riemannian density and parallel bundle metric.

\begin{proposition}
The operator $D_{J,L}$ is a first-order elliptic Dirac-type
operator, with square the flatly twisted Hodge Laplacian
$\Delta_{J,L}$.  Its $L^2$ realisations require additional domain
conditions if $M_J$ is incomplete.
\end{proposition}
\begin{proof}
In a local parallel orthonormal frame of $L$, these operators
identify with the untwisted operators of
Theorem~\ref{thm:deRhamDirac}; their formulas are preserved by
the constant sign changes on overlaps.
\end{proof}

There is no conclusion that $D_{J,L}$, for varying $J$, assembles
to an operator on $X$: the obstruction of
Example~\ref{ex:laplace-obstruction} remains, already for a trivial
local system.

\subsection{Gerbes and non-abelian extensions}
For a suitable smooth site, an abelian $A$-banded gerbe is described
by degree-two \v{C}ech data with values in $A$.  In particular, a
$\mathbb Z_2$-local system is a degree-one object and is \emph{not}
the same as a $\mathbb Z_2$-banded gerbe.  Under the usual hypotheses,
$U(1)$-bundle gerbes are instead associated with a degree-three
integral Dixmier--Douady class.  A gerbe with connection can carry
a closed $3$-form curvature, but such a curvature is neither
constructed nor determined by a sign local system in this paper.

Group extensions, lifting obstructions, and gerbes are related when
one starts with \emph{specified} extension data and a principal
bundle; see \cite{Neeb2007,BreenMessing2005,Murray1996}.
Without that additional data, neither an extension class nor a
gerbe curvature follows from the present spherical construction.
These observations delimit the higher-geometric applications of
Part~II.

\section{Lie groups: connections, classification, and examples}
\label{sec:classical-groups}

Throughout the connection and examples of this section $G$ is a Fr\'echet--Lie group,
$E=E^{\mathrm{sph}}(G)$, and
\[
 B_G=E/G,\qquad B_{G,\pm}=E/(G\times\mathbb Z_2).
\]
Both quotients carry their quotient diffeologies.  The second space
will also be called the projective double quotient; its classifying
properties are \emph{not} inferred from this terminology.

\subsection{The actual principal bundles}
\begin{proposition}[Local triviality]
\label{prop:spherical-principal}
The quotient projection $E\to B_G$ is a diffeological principal
$G$-bundle for the diagonal left action.  Similarly,
$E\to B_{G,\pm}$ is a principal $(G\times\mathbb Z_2)$-bundle.
\end{proposition}
\begin{proof}
Each $x_i:E\to\mathbb R$ is smooth, since on the generating
chart $S_I\times G^I$ its pullback is the usual coordinate if
$i\in I$, and zero otherwise. The open saturated sets
$V_i=\{x_i\ne0\}$ cover $E$. The map
$\gamma_i:V_i\to G$, $[(x,g)]\mapsto g_i$, is well-defined
because $i$ is active and smooth because every quotient plot
locally lifts. Define
\[
 N_i(z)=\gamma_i(z)^{-1}\cdot z,\qquad
 s_i([z])=N_i(z).
\]
If $z'=a\cdot z$ then $\gamma_i(z')=a\gamma_i(z)$,
so $N_i(z')=\gamma_i(z)^{-1}a^{-1}az=N_i(z)$.
Thus $N_i$ is constant along orbits and descends to a smooth
section $s_i:V_i/G\to V_i$ by the quotient universal property.
The maps
\[
 (V_i/G)\times G\longrightarrow V_i,
 \quad (b,a)\longmapsto a\cdot s_i(b),
 \qquad
 z\longmapsto\bigl([z],\gamma_i(z)\bigr)
\]
are smooth inverses. These give the asserted $G$-bundle
trivialisations. On the quotient by $G\times\mathbb Z_2$,
the corresponding saturated set is
$W_i=\{[z]:|x_i(z)|>0\}$; first multiply a representative by
the unique sign that makes $x_i>0$, and then apply $N_i$.
Since $\operatorname{sign}(x_i)$ is locally constant on $V_i$,
this operation is smooth and constant along $(G\times\mathbb Z_2)$
orbits. Its image is the slice $x_i>0$, $g_i=e$.
The inverse trivialization is the action of $G\times\mathbb Z_2$
on this slice, proving the second claim.
\end{proof}

\begin{proposition}[The extra projective factor]
\label{prop:trivial-group-projective}
For the trivial group $G=\{e\}$ one has
\[
 E^{\mathrm{sph}}(\{e\})=S_{\mathrm{fin}},\qquad
 B_{G,\pm}=S_{\mathrm{fin}}/\{\pm1\},
\]
namely the filtered infinite real projective space.  In
particular the projective double quotient cannot in general
be identified with the ordinary classifying space $BG$.
\end{proposition}
\begin{proof}
The group labels disappear and the only remaining operation
is the antipodal identification on the union of finite-dimensional
spheres.  The resulting space has the usual infinite real
projective homotopy type (with the final/weak topology), whereas
a model of $B\{e\}$ is contractible.
\end{proof}

\begin{remark}
The standard Magnot--Watts classifying theorem concerns
$B_{\rm MW}(G)$. Our separate classification theorem
(Theorem~\ref{thm:spherical-classification}) identifies $B_G$
as a classifying model for $G$ and $B_{G,\pm}$ as one for
$G\times\mathbb Z_2$, for any diffeological group and all
diffeological bases, within the $D$-numerable class.  This does
not identify $B_{G,\pm}$ with $BG$.
\end{remark}

\subsection{A weighted connection for the diagonal left action}
\label{subsec:spherical-connection}
For every generating chart $P_I=S_I\times G^I$, let
$\vartheta_i^R=dg_i\,g_i^{-1}$ be the \emph{right} Maurer--Cartan
form.  Set
\begin{equation}\label{eq:weighted-omega}
 \omega_I=\sum_{i\in I}x_i^2\,\vartheta_i^R.
\end{equation}
This differs from using the left Maurer--Cartan form
$g_i^{-1}dg_i$.  The choice of the right form is forced by our
convention that $G$ acts \emph{on the left}.

\begin{theorem}[Spherical barycentric connection]
\label{thm:spherical-connection}
The forms $\omega_I$ determine a smooth plotwise
$\mathfrak g$-valued $1$-form $\omega$ on $E$. It is a
principal connection for the left $G$-bundle $E\to B_G$ and,
with the discrete factor acting trivially on $\mathfrak g$,
for the $(G\times\mathbb Z_2)$-bundle $E\to B_{G,\pm}$.
Its curvature in the left-action convention is
\[
 F_\omega=d\omega-\tfrac12[\omega\wedge\omega].
\]
\end{theorem}
\begin{proof}
Two smooth lifts of a quotient plot have identical spherical
coordinates and identical group labels wherever the corresponding
coordinate is nonzero.  Therefore their pulled-back weighted
Maurer--Cartan forms agree on that open locus, and both weighted
terms vanish on the interior of the zero locus.  Equality extends
by continuity across its boundary, as in
Lemma~\ref{lem:liftindependent}.  Compatibility under finite
coordinate inclusions then gives a globally defined plotwise form.
For $a\in G$, $(L_a)^*\vartheta_i^R=\operatorname{Ad}_a\vartheta_i^R$.
If $\xi_E$ is the fundamental vector field for the left action,
$\vartheta_i^R(\xi g_i)=\xi$, whence
$\omega(\xi_E)=(\sum_i x_i^2)\xi=\xi$.
The sign action leaves $\omega$ fixed.  These are the equivariance
and reproduction axioms for a principal left connection; the
curvature sign follows from the right Maurer--Cartan convention.
\end{proof}

\begin{corollary}[Chern--Weil forms]
\label{cor:chern-weil-spherical}
Let $P:\mathfrak g^k\to\mathbb R$ be an
$\operatorname{Ad}(G)$-invariant continuous symmetric $k$-linear form
(the continuity hypothesis is essential in Fr\'echet dimension). Then
\[
 P(F_\omega,\ldots,F_\omega)
\]
is a closed basic diffeological $2k$-form and descends to both
$B_G$ and $B_{G,\pm}$.
\end{corollary}
\begin{proof}
The connection form is $G$-equivariant and reproduces
fundamental vectors. Its curvature
$F=d\omega-\frac12[\omega\wedge\omega]$ is horizontal:
Cartan's formula together with the equivariance of $\omega$
gives $\iota_{\xi_E}F=0$ for every $\xi\in\mathfrak g$.
It is $\operatorname{Ad}$-equivariant, while the extra sign
action acts trivially on $\mathfrak g$ and preserves $\omega$.
Consequently $P(F^k)$ is horizontal and invariant for both
principal actions. In any local bundle trivialization from
Proposition~\ref{prop:spherical-principal}, a horizontal
invariant form is the pullback of a form on the base. These
local forms agree on overlaps because their pullbacks agree;
thus they descend to the quotient diffeology. The Bianchi
identity in our left convention is
$dF-[\omega\wedge F]=0$, as follows by differentiating
$F=d\omega-\frac12[\omega\wedge\omega]$ and using the
Jacobi identity. Invariance of $P$ annihilates commutator terms
in $dP(F^k)=kP(dF,F^{k-1})$, hence the descended forms
are closed.
\end{proof}

\begin{remark}
These forms are characteristic forms of the \emph{displayed}
principal bundles.  The following classification theorem uses
only the bundle's local trivialisations, smooth partitions of unity,
and equivariant homotopies; it does not rely on this connection
or on regularity of the Lie group.
\end{remark}

\subsection{Classification without a regularity hypothesis}
\label{subsec:classifying-proof}
The classifying argument in this subsection uses the group laws,
local smooth trivialisations, and smoothly locally finite partitions
of unity; it does \emph{not} use integration of Lie-algebra curves or
horizontal parallel transport.  Accordingly, $G$ may here be an
\emph{arbitrary diffeological group}.  The connection constructed in
Subsection~\ref{subsec:spherical-connection} is additional structure
when $G$ is a Lie group, not a prerequisite for classification.

\begin{definition}[Smooth numerability]
A principal diffeological $H$-bundle $P\to X$ is
\emph{$D$-numerable} if it has $D$-open local trivialisations
$(U_i,s_i)$ and a subordinate partition of unity by smooth functions
$(\varphi_i)$ whose supports are locally finite in the $D$-topology.
We use the equivalent convention of
\cite[Section~4]{ChristensenWuSmooth}.
\end{definition}

\begin{proposition}[An explicit numerable universal bundle]
\label{prop:spherical-numerable}
For any diffeological group $G$ and
$H=G$ or $G\times\mathbb Z_2$, the principal bundle
\[
 E^{\rm sph}(G)\longrightarrow E^{\rm sph}(G)/H
\]
of Proposition~\ref{prop:spherical-principal} is $D$-numerable.
\end{proposition}
\begin{proof}
Put $B=E^{\rm sph}(G)/H$ and denote by $\lambda_i:B\to[0,1]$
the functions $\lambda_i([x,g])=x_i^2$.  They are smooth,
$H$-invariant, have finite support at every point, and satisfy
$\sum_i\lambda_i=1$.  The $D$-open sets
$U_i=\{\lambda_i>0\}$ trivialise the principal bundle by
Proposition~\ref{prop:spherical-principal}.

The family $(\lambda_i)$ need not itself have locally finite
supports.  We construct a new family for which this is true.
Index the coordinates by $i\geq1$, choose positive constants
$\delta_i=2^{-i-2}$, so that $\sum_i\delta_i<\frac34$, and
fix a smooth function $\chi:\mathbb R\to[0,\infty)$ which
vanishes on $(-\infty,0]$ and is strictly positive on $(0,\infty)$.
Set
\[
  S_{i-1}=\sum_{j<i}\lambda_j,\qquad
  \rho_i=\chi(\lambda_i-\delta_i)\,
          \chi\Bigl(\frac34-S_{i-1}\Bigr).
\]
Each $\rho_i$ is smooth, and its support is contained in the
closed set $\{\lambda_i\geq\delta_i\}$, hence in $U_i$.
For every $b\in B$ some $N$ satisfies $S_N(b)>\frac34$;
the same strict inequality holds on a $D$-open neighbourhood
of $b$.  On this neighbourhood $\rho_i=0$ for every $i>N$.
Thus the family $(\rho_i)$ has locally finite supports.

It has strictly positive sum everywhere: let $N$ be the first
index such that $S_N(b)\geq\frac34$.  If all $\rho_i(b)$
vanished for $i\leq N$, then $S_{i-1}(b)<\frac34$ would
force $\lambda_i(b)\leq\delta_i$ for all $i\leq N$.
This would give $S_N(b)\leq\sum_i\delta_i<\frac34$, a
contradiction.  Hence $\rho(b)=\sum_i\rho_i(b)>0$.
The smooth functions $\rho_i/\rho$ constitute a locally
finite partition of unity subordinate to $(U_i)$.
\end{proof}

\begin{lemma}[An explicit classifying map]
\label{lem:classifying-map-construction}
Let $G$ be any diffeological group and $H=G$ or
$G\times\mathbb Z_2$.  For a $D$-numerable principal
$H$-bundle $P\to X$, there is a smooth $H$-equivariant
map $f:P\to E^{\rm sph}(G)$, and its induced map
$F:X\to E^{\rm sph}(G)/H$ satisfies $P\cong F^*E$.
\end{lemma}
\begin{proof}
By \cite[Proposition~4.10]{ChristensenWuSmooth}, there is a
\emph{countable} locally finite smooth trivialising partition
of unity $(\varphi_i)$; no countability or paracompactness
assumption on $X$ is needed beyond $D$-numerability of $P$.
Choose smooth local sections $s_i$ on the trivialising opens.
For $p\in P$ over $x$, write $p=h_i(p)\cdot s_i(x)$,
with $h_i(p)=g_i(p)$ for $H=G$ and
$h_i(p)=(g_i(p),\varepsilon_i(p))$ for
$H=G\times\mathbb Z_2$.  Put
\[
  r(x)=\Bigl(\sum_i\varphi_i(x)^2\Bigr)^{1/2}>0.
\]
At coordinates with $\varphi_i(x)>0$, set
\[
 x_i(p)=\frac{\varphi_i(x)}{r(x)},\quad
 \text{label }g_i(p)\qquad(H=G),
\]
and, in the projective case,
\[
 x_i(p)=\varepsilon_i(p)\frac{\varphi_i(x)}{r(x)},
 \quad\text{label }g_i(p).
\]
At zero weight the group label is forgotten.  The local
finiteness of $(\varphi_i)$ gives, in a neighbourhood of
each point of $P$, a lift through one generating spherical
chart $S_I\times G^I$; hence $f$ is smooth for the \emph{final}
diffeology, not merely coordinatewise smooth.  The formulas
are $H$-equivariant for the left action.  They therefore
induce $F$ on the base and a smooth equivariant map
$P\to F^*E$.  In active-coordinate bundle trivialisations
this map is multiplication by a smooth group element and
has a smooth inverse.  Consequently $P\cong F^*E$.
\end{proof}

\begin{lemma}[Homotopy invariance without horizontal lifting]
\label{lem:frechet-parallel-transport}
For any diffeological group $H$, any $D$-numerable smooth
principal $H$-bundle $Q\to Y$, and smoothly homotopic maps
$F_0,F_1:X\to Y$, the bundles $F_0^*Q$ and $F_1^*Q$ are
isomorphic as smooth principal $H$-bundles.
No connection or regularity assumption is involved.
\end{lemma}
\begin{proof}
This is \cite[Corollary~4.14]{ChristensenWuSmooth}; we
recall the mechanism because it is relevant here.  Extend
the smooth homotopy to $X\times\mathbb R$ by a smooth time
reparametrisation constant near its endpoints, and pull
back $Q$.  Its $D$-numerability provides a smooth locally
finite trivialising partition.  The cylinder lemma
\cite[Proposition~4.12]{ChristensenWuSmooth} supplies a
locally finite smooth partition $(\psi_a)$ on $X$ subordinate
to a cover $(V_a)$ such that the pullback bundle is trivial
over each $V_a\times[0,1]$.

Choose smooth functions $\eta_a:X\to[0,1]$, supported in
$V_a$, such that the family is locally finite and
$\max_a\eta_a(x)=1$ for every $x$.  Such functions are obtained
from $(\psi_a)$ by the usual smooth cut-off applied to
$\psi_a/\sum_b\psi_b^2$; cf. the proof of
\cite[Proposition~4.13]{ChristensenWuSmooth}.
In each chosen trivialisation, lift the base map
\[
 r_a(x,t)=\bigl(x,(1-\eta_a(x))t+\eta_a(x)\bigr)
\]
by \emph{keeping the group coordinate constant}; extend
this lifted map by the identity where $\eta_a=0$.
The locally finite ordered composition of these lifted maps
is smooth, equivariant, and covers $(x,t)\mapsto(x,1)$,
since at least one $\eta_a(x)$ equals one.  This induces
a smooth fibrewise principal-bundle isomorphism with the
product of the endpoint bundle and $[0,1]$.  Restricting
at the endpoints gives $F_0^*Q\cong F_1^*Q$.
No differential equation on $H$ is solved.
\end{proof}

\begin{theorem}[Classification by spherical Milnor bundles]
\label{thm:spherical-classification}
Let $G$ be \emph{any diffeological group}, and let
$H=G$ or $G\times\mathbb Z_2$ act on $E=E^{\rm sph}(G)$
as in Proposition~\ref{prop:spherical-principal}.
For \emph{every} diffeological space $X$, pullback induces
a natural bijection
\[
 [X,E/H]_{\rm sm}\ \cong\
 \operatorname{Prin}^{D\text{-}\rm num}_H(X).
\]
In particular, $E/G$ classifies $D$-numerable principal
$G$-bundles and $E/(G\times\mathbb Z_2)$ classifies
$D$-numerable principal $G\times\mathbb Z_2$-bundles.
Neither regularity nor a connection is required.
\end{theorem}
\begin{proof}
By Proposition~\ref{prop:spherical-numerable}, the
universal principal $H$-bundle $E\to E/H$ is $D$-numerable.
Its pullbacks therefore remain $D$-numerable.  Conversely,
Lemma~\ref{lem:classifying-map-construction} produces a
classifying map for any $D$-numerable $H$-bundle on $X$.
This proves surjectivity.

If two classifying maps yield isomorphic principal bundles,
identify their total spaces with a common bundle $P$.
The corresponding equivariant maps $P\to E$ are smoothly
$H$-equivariantly homotopic by
Theorem~\ref{thm:equivariant-homotopy}; the induced maps
$X\to E/H$ are therefore smoothly homotopic.  Conversely,
by Lemma~\ref{lem:frechet-parallel-transport}, smoothly
homotopic classifying maps pull back the $D$-numerable
universal bundle to isomorphic bundles.  This proves
injectivity and well-definedness on homotopy classes.
Naturality is immediate from naturality of pullback.
\end{proof}

\begin{remark}[Classification versus differential geometry]
The theorem is independent of the existence of a Lie algebra,
a Maurer--Cartan form, or a regular evolution map.  The
weighted connection and Chern--Weil forms in the preceding
subsection remain additional differential-geometric data for
Lie groups satisfying their respective hypotheses.  The
$D$-numerability assumption on the \emph{bundles} is
indispensable in general; no classification of all arbitrary
non-$D$-numerable diffeological principal bundles is claimed.
\end{remark}

\subsection{Compact orthogonal and unitary groups}
For $G=O(n)$, $SO(n)$ or $SU(n)$, the Lie algebras admit
$\operatorname{Ad}(G)$-invariant positive inner products:
\[
 h_{\mathfrak o(n)}(X,Y)=-\operatorname{tr}(XY),\qquad
 h_{\mathfrak{su}(n)}(X,Y)=-\operatorname{Re}\operatorname{tr}(XY).
\]
The corresponding bi-invariant Riemannian metrics on $G$
therefore satisfy the assumptions of Section~\ref{6}.
On each support stratum $E_J^\circ$ we have the genuine metric
\[
 g_J^\circ=\sum_{j\in J}dx_j^2
   +\sum_{j\in J}x_j^2 h_G^{(j)}.
\]
It is invariant under the diagonal left $G$-action and the
global sign action; the associated stratumwise de Rham--Dirac
operator of Section~\ref{sec:dirac-strata} is equivariant.
The heat-kernel construction of Corollary~\ref{cor:hilbert-compact}
does provide a non-canonical smooth positive ambient tangent
metric for compact $G$. It does not, by itself, provide a
global Hodge adjoint or Dirac operator on $E$ or its orbit spaces.

The inclusions $SO(n)\subset O(n)$ and
$\operatorname{Spin}(n)\to SO(n)$ describe ordinary reductions
and lifts of \emph{given} bundles.  In a setting where an
associated oriented rank-$n$ bundle $V$ is defined, its spin
lift is obstructed by $w_2(V)$.  A spin lift of $V$ does not
by itself construct an ambient Dirac operator on the spherical
quotient.  Likewise, the invariant polynomial in
Corollary~\ref{cor:chern-weil-spherical} yields the usual
Chern--Weil expressions for a chosen $SU(n)$-connection, but
no unproved universality assertion is required.

\subsection{The connected Lorentz group and two metric choices}
For $n\ge2$, let $G=SO^+(1,n)$ be the identity component of
the real Lorentz group, with Lie algebra $\mathfrak{so}(1,n)$.
The Killing form $\kappa$ is non-degenerate and indefinite.
It defines a left-invariant pseudo-Riemannian tensor, whose
weighted analogues on the full-support strata are also
pseudo-Riemannian; the \emph{positive} Hodge and elliptic
Dirac statements of Sections~\ref{sec:hodge-strata}--
\ref{sec:dirac-strata} do not apply to this choice.

A Cartan involution $\theta$ instead defines the positive
inner product $h_\theta(X,Y)=-\kappa(X,\theta Y)$.
Its left-invariant extension to $SO^+(1,n)$ is Riemannian,
and hence yields precisely the plotwise energies and
stratumwise elliptic operators proved earlier.  It is not
bi-invariant in general, but \emph{left invariance is enough}
for equivariance under the action used in this manuscript.

The double cover $\operatorname{Spin}^+(1,n)\to SO^+(1,n)$
is relevant for Lorentzian spin geometry of a separately
specified pseudo-Riemannian manifold or vector bundle.
It does not turn the spherical total space into a globally
Lorentzian spin manifold, nor does it remove the
support-change obstruction for Dirac operators.

\section{Diffeomorphism groups, fiber bundles, and defects}
\label{sec:diff-classifying}

Let $M$ be a compact finite-dimensional smooth manifold and
$G=\operatorname{Diff}(M)$ with its usual Fr\'echet or functional
diffeology.  The principal bundle in the classical Milnor
construction must be distinguished from the auxiliary spherical
principal bundles in Proposition~\ref{prop:spherical-principal}.
In this section all statements on smooth families assume the
indicated local trivialisations, connections, and parameter
regularity; no infinite-dimensional Hodge or Dirac operator is
inferred solely from the geometry of $G$.

\subsection{Classification and spherical constructions}
The standard diffeological Milnor model classifies
$D$-numerable principal $\operatorname{Diff}(M)$-bundles under
the hypotheses of Theorem~\ref{thm:MW}.  An associated smooth
$M$-bundle is formed from such a principal bundle via the
usual action of $\operatorname{Diff}(M)$ on $M$.

The spherical space $E^{\mathrm{sph}}(\operatorname{Diff}(M))$
comes with a left principal $\operatorname{Diff}(M)$-bundle
\[
 E^{\mathrm{sph}}(\operatorname{Diff}(M))\longrightarrow
 B_G=E^{\mathrm{sph}}(\operatorname{Diff}(M)) /\operatorname{Diff}(M),
\]
and with an additional commuting sign symmetry.  These are actual bundles and their total space is smoothly
contractible by Theorem~\ref{thm:spherical-contract}. For compact $M$, the Fr\'echet--Lie group
$\operatorname{Diff}(M)$ is $C^\infty$-regular~\cite{KrieglMichorRegular1998}
and smoothly point-separating: evaluations $\varphi\mapsto f(\varphi(m))$
with $m\in M$ and $f\in C^\infty(M,\mathbb R)$ distinguish
diffeomorphisms. Thus Theorem~\ref{thm:spherical-classification}
applies to its $D$-numerable bundles over the stated class of bases.  The observation in Proposition~\ref{prop:trivial-group-projective}
shows why the double quotient concerns the larger structure group.

\subsection{Reductions and fiberwise geometric data}
Let $H\subset\operatorname{Diff}(M)$ be a specified subgroup
preserving a geometric structure $\mathcal S$ on $M$.
A reduction of an $M$-bundle's principal
$\operatorname{Diff}(M)$-bundle to $H$ yields a consistently
glued fiberwise $\mathcal S$-structure.  Conversely, when the
space of structures of type $\mathcal S$ is the relevant
homogeneous $\operatorname{Diff}(M)/H$-space and its smooth
local trivialisations are available, such fiberwise structures
are equivalent to reductions of structure group.  A lift of the
classifying map to $BH$ expresses this reduction \emph{within an
established classifying theory}.  In particular, a reduction
to the isometry group of a fixed metric is sufficient, but not
necessary, for the existence of an arbitrary smoothly varying
family of fiber metrics.

\begin{proposition}[Vertical Dirac family under explicit hypotheses]
\label{prop:vertical-Dirac}
Let $\pi:Y\to B$ be a smooth bundle of compact spin
manifolds of fixed dimension.  Suppose its vertical tangent
bundle has a smooth Riemannian metric, a specified smoothly
varying vertical spin structure and an associated vertical
spinor bundle $S^{\mathrm{vert}}$.  Then the vertical
Levi--Civita spin connection defines a family
\[
 D_b:\Gamma(Y_b,S^{\mathrm{vert}}|_{Y_b})
 \longrightarrow\Gamma(Y_b,S^{\mathrm{vert}}|_{Y_b}),
 \qquad b\in B,
\]
of first-order elliptic Dirac operators.  For ordinary
finite-dimensional smooth bases the family depends smoothly
on $b$ in local trivialisations; diffeological bases require
the corresponding plotwise smoothness assumptions.
\end{proposition}
\begin{proof}
In a smooth local trivialisation, the vertical metrics and
spin frames depend smoothly on the base parameters.  The
Levi--Civita and induced spin connection coefficients therefore
depend smoothly on those parameters.  Clifford contraction
gives a locally smooth family of first-order operators.
The specified vertical spinor bundle provides precisely the
compatibility needed for these expressions to agree on
overlaps.
\end{proof}

\begin{remark}
The existence of a fiberwise spin structure is not, by
itself, a canonical choice of fiberwise Riemannian metric.
Nor does Proposition~\ref{prop:vertical-Dirac} claim a spin
structure on the total space $Y$ or a Dirac operator on
$B_G$ or $B_{G,\pm}$.
\end{remark}

\subsection{Transport of operators and the defect cocycle}
A Dirac operator on a fixed spin manifold cannot be conjugated
by an arbitrary diffeomorphism without identifying the spinor
bundles and their metrics.  We therefore formulate a precise
defect calculation for a group $H$ acting by invertible maps
$U_h$ on a common space of sections $\mathcal V$.
Assume $U_{hk}=U_hU_k$, and let $D$ be a linear operator with
an $H$-invariant domain (or work formally with differential
operators on a natural bundle).  Put
\[
 c_D(h)=U_h D U_h^{-1}-D.
\]

\begin{proposition}[Defect cocycle]
\label{prop:defect-cocycle}
For $h,k\in H$ one has
\begin{equation}\label{eq:defect-cocycle}
 c_D(hk)=c_D(h)+U_hc_D(k)U_h^{-1}.
\end{equation}
If $H$ acts differentiably and $X$ belongs to its Lie algebra,
then, with $\mathcal L_X=\left.\frac d{dt}\right|_0U_{\exp(tX)}$,
\[
 \left.\frac d{dt}\right|_0c_D(\exp(tX))
 = [\mathcal L_X,D].
\]
\end{proposition}
\begin{proof}
Insert $U_h^{-1}U_h$ between the two conjugations:
\[
 U_{hk}DU_{hk}^{-1}-D
 =U_h(U_kDU_k^{-1}-D)U_h^{-1}
  +(U_hDU_h^{-1}-D).
\]
Differentiation of $U_{\exp(tX)}DU_{\exp(tX)}^{-1}$ at zero
gives the commutator, under the stated domain conditions.
\end{proof}

\begin{remark}[Cocycle versus obstruction]
The map $c_D$ is an algebraic group $1$-cocycle with values in
operators under conjugation.  It measures failure to preserve a
\emph{fixed} operator, not necessarily failure of a Dirac family
to exist.  It does not, without further construction, define a
finite-rank ``defect bundle'', a characteristic class, a
determinant line, or a gerbe.  Establishing such objects requires
analytic control of operator domains, trace ideals or suitable
Fredholm families, and specified transition and gluing data.
\end{remark}

\subsection{Scope of the applications}
The results established here consist of the principal bundles
of Proposition~\ref{prop:spherical-principal}, the finite-group
model geometries, the conditional vertical spin Dirac families,
and the explicit defect cocycle.  Infinite-dimensional
connections, determinant gerbes and Schwinger cocycles are
natural directions for further work, but they are not consequences
of the chartwise spherical energies or of the stratumwise
operators proved above.

\appendix
\section{Square-root coordinates and Fisher--Rao geometry}
\label{app:fisher-rao}

\subsection{Interior formula}
Let $I$ be finite and nonempty and
\[
 \Delta_I^\circ=\{t_i>0,\ \sum_{i\in I}t_i=1\}.
\]
With one standard convention for Fisher--Rao normalisation, its
metric is
\[
 g_{\rm FR}=\sum_{i\in I}\frac{dt_i\otimes dt_i}{t_i}
 \quad\text{restricted to }T\Delta_I^\circ.
\]
On each open sign chamber of
$S_I^\circ=\{x\in S_I:x_i\ne0\ \forall i\}$, the map
$q_I(x)_i=x_i^2$ is a diffeomorphism onto
$\Delta_I^\circ$.  Its pullback is
\begin{equation}\label{eq:fisher-rao-pullback}
 q_I^*g_{\rm FR}=4\sum_{i\in I}dx_i\otimes dx_i
 \quad\text{on }TS_I^\circ.
\end{equation}
Indeed $dt_i=2x_i\,dx_i$, and the division by $t_i=x_i^2$
is legitimate on this open chamber.

\subsection{Extension on the sphere, not on the simplex}
The right-hand side of~\eqref{eq:fisher-rao-pullback}
extends smoothly as a positive-definite tensor to the whole
finite sphere $S_I$.  This is a \emph{regular extension on the
signed spherical model}.  It does not say that
$g_{\rm FR}$ extends to a smooth nondegenerate tensor on the
closed simplex with its subset diffeology, nor that the map
$q_I$ has a smooth inverse at a vanishing coordinate.
Moreover, after attaching group labels and forgetting those
at zero coordinates, the full internal tangent requires
additional analysis (Section~\ref{5}).

\subsection{Second-order barycentric behaviour}
For a smooth spherical curve through $x_k(0)=0$, define
$t_k(s)=x_k(s)^2$.  Then
\[
 t_k'(0)=0,\qquad t_k''(0)=2(x_k'(0))^2.
\]
Thus activation of an inactive spherical coordinate is
invisible to the \emph{first differential of the squaring map};
this does not make the corresponding spherical tangent vector
zero, as the branch example in Section~\ref{5} emphasises.

\subsection{Group-coordinate coupling}
For a finite-dimensional Lie group with left-invariant
Riemannian metric $h_G$, the plotwise energy of
Section~\ref{6} is generated by
\[
 b_I=\sum_{i\in I}dx_i^2+\sum_{i\in I}x_i^2h_G^{(i)}.
\]
Its spherical term is one quarter of the pullback Fisher--Rao
metric on each open sign chamber.  The group term vanishes
on an inactive label, but the vanishing alone does not yield
a Riemannian metric on the ambient internal tangent, or
a globally compatible Hodge Laplacian.  The counterexample
of Section~\ref{sec:hodge-strata} makes the latter distinction
substantive, rather than merely terminological.

\subsection{Perspective}
The signed spherical replacement gives a geometrically useful
regularisation of the \emph{interior Fisher--Rao expression}.
Its interaction with quotient-diffeological branching and
global analytic operators requires further structure beyond
this finite-dimensional identity.

\vskip 12pt
\paragraph{Conflict of interest statement:} The author declares no conflict of interest.
\vskip 12pt
\paragraph{Funding statement:} No funding.
\vskip 12pt
\paragraph{Acknowledgements:}
J.-P.M. thanks the France 2030 framework programme Centre Henri Lebesgue
ANR-11-LABX-0020-01 for creating an attractive mathematical environment.
\vskip 12pt
\paragraph{Declaration of generative AI and AI-assisted technologies in the writing process}
During the preparation of this work the author used ChatGPT and Mistral AI
for language editing and drafting assistance. The author reviewed and
edited the content and takes responsibility for the resulting manuscript.

\end{document}